\documentclass[
11pt, % The default document font size, options: 10pt, 11pt, 12pt
english, % ngerman for German
onehalfspacing, % Single line spacing, alternatives: onehalfspacing or doublespacing
headsepline, % Uncomment to get a line under the header
]{amsproc} % The class file specifying the document structure
\usepackage[utf8]{inputenc} % Required for inputting international characters
\usepackage[T1]{fontenc} % Output font encoding for international characters

\usepackage[backend=bibtex,style=numeric,natbib=true,sorting=nyt]{biblatex}
\usepackage{siunitx}
\usepackage[autostyle=true]{csquotes} % Required to generate language-dependent quotes in the bibliography
\usepackage{graphicx}              % to include figures
\usepackage{amsmath}               % great math stuff
\usepackage{amsfonts}              % for blackboard bold, etc
\usepackage{amsthm}                % better theorem environments
\usepackage{amssymb}
\usepackage[arrow, matrix, curve]{xy}
\usepackage{tikz}
\usepackage{appendix}
\usepackage{mathtools}% better theorem environments
\usetikzlibrary{cd}

\newtheorem{thm}{Theorem}[subsection]

\newtheorem*{thmo}{Theorem}
\newtheorem*{lemo}{Lemma}
\newtheorem*{coro}{Corollary}

\newtheorem{lem}[thm]{Lemma}
\newtheorem{prop}[thm]{Proposition}
\newtheorem{cor}[thm]{Corollary}

\theoremstyle{definition}

\newtheorem{defi}[thm]{Definition}

\newtheorem{exam}[thm]{Example}
\newtheorem{rema}[thm]{Remark}

\newcommand{\colim}{\operatornamewithlimits{{colim}}}
\newcommand{\hocolim}{\operatornamewithlimits{{hocolim}}}
\newcommand{\holim}{\operatornamewithlimits{{holim}}}
\newcommand{\RR}{\mathbb{R}}      % for Real numbers
\newcommand{\Hom}{\operatorname{Hom}}
\newcommand{\Map}{\operatorname{Map}}

\newcommand{\RHom}{\operatorname{RHom}}
\newcommand{\RMap}{\operatorname{RMap}}
\newcommand{\RNat}{\operatorname{RNat}}
\newcommand{\Aut}{\operatorname{Aut}}
\newcommand{\Fun}{\operatorname{Fun}}
\newcommand{\hofib}{\operatorname{hofib}}
\newcommand{\Sc}{\operatorname{Sc}}

\newcommand{\Free}{\operatorname{Free}}
\newcommand{\Conf}{\operatorname{Conf}}
\newcommand{\Id}{\operatorname{Id}}

\newcommand{\free}{\operatorname{free}}

\newcommand{\sdSet}{\textup{\textsf{sdSet}}}
\newcommand{\sodSet}{\textup{\textsf{scdSet}}}
\newcommand{\odSet}{{\textup{\textsf{cdSet}}}}
\newcommand{\sSetOp}{\textup{\textsf{sSetOp}}}
\newcommand{\sSet}{\textup{\textsf{sSet}}}
\newcommand{\Set}{\textup{\textsf{Set}}}
\newcommand{\dSet}{\textup{\textsf{dSet}}}

\renewcommand{\lg}{\langle}
\newcommand{\rg}{\rangle}
\newcommand{\bound}{\textup{bound}}
\newcommand{\cobound}{\textup{cobound}}
\newcommand{\corl}{\textup{c}}
\newcommand{\ccorl}{\textup{cc}}
\newcommand{\co}{\colon}
\newcommand{\op}{\textup{op}}
\newcommand{\cl}{\textup{cl}}
\newcommand{\smin}{\smallsetminus}

\newcommand{\sh}{\textup{sh}}
\newcommand{\acat}{\textup{\textsf{V}}_n}
\newcommand{\bcat}{\textup{\textsf{W}}_n}
\newcommand{\rran}{\textup{RRan}}
\newcommand{\rlan}{\textup{RLan}}
\newcommand{\inco}{\textup{in}}

\AtBeginBibliography{\scriptsize}

\title{A spectral sequence for spaces of maps between operads}
\author{Florian G\"oppl and Michael Weiss}

\begin{document}

\begin{abstract} Under mild conditions on topologically enriched operads $P$ and $Q$, the derived mapping space
$\RHom(P,Q)$ is the limit (sequential homotopy inverse limit) of a tower whose $n$-th layer admits a description in terms of
certain (small) diagrams
$J_n(P)$ and $J_n(Q)$. More precisely $J_n(P)$ is a 3-term diagram of spaces with action of $\Sigma_n$, of the form
\[ \bound_n(P)\to P(n)\to \cobound_n(P) \]
where $P(n)$ is the space of $n$-ary operations in $P$. The statement takes some inspiration from manifold
calculus, but the proof relies on the homotopical theory of dendroidal spaces and the concept of dendroidal nerve
of an operad.
\end{abstract}

\maketitle

\section{Introduction} \label{sec-intro} % Main chapter title

\label{Introduction} % Change X to a consecutive number; for referencing this chapter elsewhere, use \ref{ChapterX}

Operads are tools well-suited to describe and classify additional algebraic structures on objects in symmetric monoidal categories.
On the other hand they are a natural generalization of (enriched) categories allowing morphisms to have any finite number of sources.
Operads (in a more restrictive one-object setting) were first defined by Peter May
in \cite{May72}. Closely related notions can be seen in the earlier book by Boardman-Vogt \cite{BV68} and a specific
operad emerged earlier still in the work of Stasheff \cite{Stash} and Sugawara \cite{Sug}. For a
very readable survey and exposition see \cite[\S2]{Ada}.
In his book, May proved the famous recognition principle which gives an ``operadic'' characterization of based spaces which are
homotopy equivalent to some $n$-fold loop space.
Operads have since appeared in various branches of mathematics and mathematical physics.
The principal aim of this investigation was to find a way to understand spaces of maps between operads.
The ``little disk'' operads are important examples and test cases. \\
We will do this by translating the problem into the language of dendroidal spaces. These are contravariant functors from a certain category $\Omega$ of trees
to the category $\sSet$ of simplicial sets.
The theory of dendroidal sets and dendroidal spaces was introduced by Ieke Moerdijk and
Ittay Weiss in \cite{MW07} (see also \cite{W07})
and further investigated by Cisinski and Moerdijk in \cite{CM09,CM11,CM13}.
A simplicially enriched operad $P$ determines a dendroidal space $N_dP$, known as the dendroidal nerve of $P$. There is a
map of derived mapping spaces from $\RHom(P,Q)$ to $\RHom(N_dP,N_dQ)$ which  is a weak equivalence in the cases we are interested in.  \\
Although derived mapping spaces have a standard description in the context of model categories, we will mostly avoid this
description and rely on the description due to Dwyer-Kan instead \cite{DK80b}. They construct derived mapping spaces for objects in any category $C$ equipped
with a wide subcategory $W$ (whose morphisms play the role of weak equivalences). If $W$ happens to be the subcategory of weak equivalences
in a model category, then these constructions yield weakly equivalent results. For our purposes a morphism of dendroidal
spaces is a weak equivalence if and only if it is a levelwise
equivalence of simplicial sets. The goal, then, is to understand the homotopy type of derived mapping spaces between dendroidal
nerves of (some) operads. \\
We approach this problem by mapping the space $\RHom(N_dP,N_dQ)$ to a tower of derived
mapping spaces obtained by restricting $N_dP$ and $N_dQ$ to certain subcategories $\Omega\lg k\rg$,
where $0 \leqslant k < \infty$.
The subcategory $\Omega\lg k\rg$ of $\Omega$ is the full subcategory on trees with vertices of valence $\leqslant k+1$ only (to put it
differently, trees in which no vertex has more than $k$ \emph{incoming} edges).
We note that these categories $\Omega\lg k\rg$ are closed under grafting of trees. Contravariant functors from $\Omega\lg k\rg$ to
the category of simplicial sets
will be called $k$-truncated dendroidal spaces. A morphism of truncated dendroidal spaces is a weak equivalence if it is a levelwise weak equivalence
of simplicial sets. With these definitions it is clear that the restriction functor $U_k$ from dendroidal spaces to $k$-truncated ones preserves
all weak equivalences and thus induces maps $\RHom(X,Y) \rightarrow \RHom(U_kX,U_kY)$ for all dendroidal spaces $X$ and $Y$.
We arrange these maps in a tower
 \[
\xymatrix@R=15pt{
 && \rule{12mm}{0mm}\vdots\rule{12mm}{0mm} \ar[d]  \\
 && \RHom(U_3X,U_3Y) \ar[d] \\
 && \RHom(U_2X,U_2Y) \ar[d] \\
\RHom(X,Y) \ar[rr] \ar@/^1.1pc/[urr] \ar@/^1.8pc/[uurr] \ar@/^2.6pc/[uuurr] && \RHom(U_1X,U_1Y)
}
\]
of derived mapping spaces. \\
In section~\ref{subsec-towconst} we set up a ``d\'evissage'' mechanism for proving homotopical
statements in categories of contravariant
functors (with values in $\sSet$) with levelwise weak equivalences. We show that every functor admits a weak equivalence from a free CW-functor, a more
restrictive instance of the concept of CW-functor in Dror-Farjoun \cite{DF}. These are functors admitting a CW-type decomposition
into cells of the shape $\Hom(-,c) \times \Delta[k]$.
We make use of this approximation to prove certain homotopical properties for contravariant functors. More precisely we show
that we can verify such a property by showing that it holds for representable functors and that it persists under
formation of homotopy pushouts and disjoint unions.
Our first application of this principle is the following statement. (Admittedly this is unsurprising, and it might have shorter
proofs and might be regarded as obvious by some readers.)
\begin{lemo} \emph{(= Lemma~\ref{TowerConvergence}, Corollary~\ref{corfilt}.)}
 The above tower of derived mapping spaces converges, i.e., for all dendroidal spaces $X$ and $Y$ the induced map
\[\RHom(X,Y) \to \holim_k \RHom(U_kX,U_kY) \]
is a weak homotopy equivalence.
\end{lemo}
Under additional assumptions on our objects the homotopy fibers of this tower can be simplified.
A dendroidal Segal space is called $1$-reduced if its values on the trivial tree and
the $0$-corolla are points and its value on the $1$-corolla is contractible. The most important example is the dendroidal nerve of a
$1$-reduced simplicially enriched operad $P$. These are operads $P$ that only have one object
and satisfy $P(0) = *$ and $P(1) \simeq *$. This notion still captures our most important examples
since all $E_n$ operads and many more are $1$-reduced.
In this setting we can, instead of working in the category $\sdSet$ of dendroidal spaces, restrict attention to
an easier category $\sodSet$. This is
based on a category of \emph{closed} trees $\Omega_{\cl}\subset\Omega$. \\
Using this model we define operadic boundary and coboundary objects reminiscent of the latching and matching objects, respectively,
from the theory of Reedy categories. Let $\ccorl_k$ be the closed $k$-corolla, an important object of $\Omega_{\cl}$ with a preferred
action of $\Sigma_k$. Evaluating objects in $\sodSet$ at $\ccorl_k$ gives a functor
\[ X\mapsto X_{\ccorl_k} \]
from $\sodSet$ to the category of simplicial sets
with $\Sigma_k$-action. (If $X=N_dP$ where $P$ is a $1$-reduced operad, then $X_{\ccorl_k}\simeq P(k)$.) We define two more functors
$\bound_k$ and $\cobound_k$ from $\sodSet$ to simplicial sets with $\Sigma_k$-action, and
natural $\Sigma_k$-maps
\[ \bound_k X \to X_{\ccorl_k} \to \cobound_k X\,. \]
Let $J_k(X)$ be the diagram just above, $\bound_k X \to X_{\ccorl_k} \to \cobound_k X$, and let $\partial J_k(X)$ be the
shorter diagram
\[ \bound_k X \to \cobound_k X \]
obtained by composing the two arrows in $J_k(X)$.
Both of these are understood to be diagrams in the category of simplicial sets with an action of the symmetric group $\Sigma_k$.
(For the present purposes a morphism of simplicial sets with $\Sigma_k$-action will be regarded as a weak
equivalence if the underlying morphism in $\sSet$ is a weak equivalence.)
\begin{thmo} \emph{(= Theorem~\ref{mainthm}.)}
Let $X$ and $Y$ be $1$-reduced dendroidal Segal spaces, to be viewed as objects of $\sodSet$.
Then the following square is a homotopy pullback square:
\[
\begin{tikzcd}
\RHom(U_kX,U_kY) \arrow[d] \arrow[r] & \RNat(J_kX,J_kY) \arrow[d]  \\
\RHom(U_{k-1}X,U_{k-1}Y) \arrow[r]& \RNat(\partial J_kX, \partial J_kY)
\end{tikzcd}
\]
\end{thmo}
(See also remark \ref{remarkThomas} for a slightly different formulation.)
This allows us to make some homotopical computations with $\RHom(E_n,E_{n+d})$.

\begin{thmo}
 The homotopy fiber of
 \[
 \RNat(U_kE_n, U_kE_{n+d}) \rightarrow \RNat(U_{k-1}E_n,U_{k-1}E_{n+d})
 \]
 is $((k-1)(d-2)+1)$-connected.
\end{thmo}

Combining this computation with the first result gives us the following estimate for the connectivity of the space $\RHom(E_n,E_{n+d})$.

\begin{coro}
 Assume $d \geq 2$. The derived mapping space $\RHom(E_n,E_{n+d})$ is $(d-1)$-connected. The spaces
 $\RHom_k(U_kE_n,U_kE_{n+d})$ are $(d-1)$-connected as well, for all $k\ge 1$.
\end{coro}

The above theorem and the corollaries are reminiscent of fundamental results in the manifold calculus and can also be used in this context.
A first application can be found in \cite[\S3.4,\S5]{W16}.

\smallskip
The operadic boundary and coboundary objects have also been investigated and used
in \cite{FTW18} and \cite{Heu18} in slightly different settings. Similar constructions can be seen in \cite{Thu}.

\medskip
\emph{Authorship.} Apart from minor revisions, this is the PhD thesis of Florian G\"oppl
(PhD degree in 2019 at WWU M\"unster). The PhD supervisor at the time was M.Weiss. Florian G\"oppl is no longer active in topology research,
but his thesis was well received and as time went by, the case for publishing it became stronger, not weaker.
It fell to M.W. to revise and submit the work and act as corresponding author, although he is hardly an author or co-author
of the article.

\smallskip
\emph{Acknowledgment.} We are indebted to Thomas Nikolaus for some helpful suggestions.

\section{Operads and dendroidal objects}   \label{sec-opdend} % Main chapter title

The purpose of this first section is to explain and motivate the notion of an operad. (The section is not a self-contained
introduction to the homotopy theory of operads and dendroidal objects.)
In the first part of the section we will give a short exposition of the basic definitions and theorems.
The second part is devoted to some closely related notions more approachable by homotopical methods.
The theory of dendroidal sets was introduced by Ieke Moerdijk and Ittay Weiss in \cite{MW07}. A dendroidal object is a
contravariant functor on an indexing category of trees. An important subclass of trees are the linear ones and
contravariant functors on this subcategory are simplicial objects.
Most of the homotopical constructions for simplicial spaces generalize to the dendroidal setting.
Our focus will be on the dendroidal analogue of (complete) Segal spaces \cite{Rez01}. \\
Throughout this article we will make some use of the theory of model categories.
A model structure on a bicomplete category $\textsf{C}$ is defined by a
triple $(\mathcal{C}o, \mathcal W, \mathcal{F}i)$ of wide subcategories of $\textsf{C}$.
(A subcategory is called wide if it contains all identity morphisms.)
These classes have to satisfy certain lifting properties analogous to the
cofibrations, weak equivalences and fibrations of topological spaces.
Although we are mostly interested in derived mapping spaces and these only depend on a class of weak equivalences, the
additional structure given by fibrations and cofibrations provides useful tools for computations of
mapping spaces and derived functors.

%----------------------------------------------------------------------------------------
%	SECTION 1
%----------------------------------------------------------------------------------------

\subsection{Operads}

\begin{defi}
 An \emph{operad} $P$ consists of a set of objects $\{x_i\}$ and for every $(n+1)$-tuple $(x_1, \ldots, x_n; x)$ of objects
 a set of morphisms $P(x_1, \ldots, x_n; x)$ subject to the following axioms:
 \begin{itemize}
 \item A morphism $id_x \in P(x;x)$ called \emph{identity of $x$}.
 \item An associative composition morphism
 \[
 \begin{tikzcd}
 P(y_1, \ldots, y_n; z) \times P(x_{1,1}, \ldots, x_{1,k_1}; y_1) \times \ldots \times P(x_{n,1}, \ldots, x_{n,k_n}; y_n) \arrow[d] \\ P(x_{1,1}, \ldots x_{n,k_n}; z).
 \end{tikzcd}
 \]
 \item For every $(n+1)$-tuple $(x_1, \ldots x_n; y)$ and every $\sigma \in \Sigma_n$ a bijection
 \[
 \sigma^* \colon P(x_1, \ldots x_n; y) \rightarrow P(x_{\sigma(1)}, \ldots x_{\sigma(n)}; y)
 \]
 respecting the other structure.
\end{itemize}
A more complete definition is given in \cite[1.1]{BM07}. \\
A morphism of operads $f \colon P \rightarrow Q$ consists of a map between objects \[ ob(P) \rightarrow ob(Q) \] and structure preserving maps
\[ P(x_1, \ldots, x_n; x) \rightarrow Q(f(x_1), \ldots, f(x_n); f(x)). \]
An operad is called \emph{monochromatic} if it has only one object.
\end{defi}

The operadic (multi-)composition can be understood more easily by picturing the multimorphisms as special (planar) trees
(so called \emph{corollas}) with several ``input'' edges and a unique ``output'' edge. Names of objects (sources and target) should be attached
to edges and the name of the multimorphism
can be attached to the unique vertex as a label. Grafting leads to a more complicated tree with several labels.
The following depicts the grafting of a $2$-morphism $u$ with two $3$-morphisms $v$ and $w$ in the monochromatic case,
where the labeling of edges with objects is unnecessary.
\[
\begin{tikzpicture}
\draw (0,2) --(1,1) node[at end,left]{$u$}  ;
\draw (2,2) --(1,1) node[midway,right] {};
\draw (1,1) --(1,0) node[midway,left] {};
\draw[fill] (1,1) circle [radius=0.07];
\draw (3,2) --(4,1) node[at end,left]{$v$} ;
\draw (5,2) --(4,1) node[midway,right]{};
\draw (4,1) --(4,0) node[midway,right]{};
\draw (4,2) --(4,1) node[midway,left]{};
\draw[fill] (4,1) circle [radius=0.07];
\draw (6,2) --(7,1) node[at end,left] {$w$} ;
\draw (7,1) --(7,0) node[midway,left]{};
\draw (8,2) --(7,1) node[midway,right] {};
\draw (7,2) --(7,1) node[at start,above]{};
\draw[fill] (7,1) circle [radius=0.07];

\draw (9.3,3) --(10,2) node[at end,left] {$v$};
\draw (10,3) --(10,2) node[at start,above] {};
\draw (10.6,3) --(10,2) node[midway,right] {};
\draw (12,2) --(11,1) node[midway,left] {};
\draw (10,2) --(11,1) node[at end,left] {$u$};
\draw (11,1) --(11,0) node[midway,left] {};
\draw (11.3,3) --(12,2) node[at end,left] {$w$};
\draw (12,3) --(12,2) node[at start,above] {};
\draw (12.6,3) --(12,2) node[midway,right] {};
\draw[fill] (10,2) circle [radius=0.07];
\draw[fill] (11,1) circle [radius=0.07];
\draw[fill] (12,2) circle [radius=0.07];
\draw[thick,->] (8.5,1.5) -- (9,1.5) node[anchor=north west] {};
\draw (2.5,1) circle [radius=0.1];

\end{tikzpicture}
\]
It is up to the operadic structure to ``simplify'' the complicated tree in the right-hand side of the picture to a
6-corolla with a single label at the unique vertex. The simplification can be thought of as something induced
(contravariantly) by a morphism from the 6-corolla to the complicated tree with three vertices and 9 edges.
(In this context it is convenient to think of trees as partially ordered sets of edges. A morphism of trees is
given by an order preserving map of edge sets, subject to additional conditions which will be
spelled out below.)

\smallskip
For every symmetric monoidal category $\textsf{C}$ it makes sense to speak of operads enriched over $\textsf{C}$. These still have a (discrete) set of objects.
A \emph{topological operad} is an operad enriched over the category of compactly generated weak Hausdorff spaces. We will more ambiguously speak
of operads enriched over spaces to mean either topological operads or operads enriched over simplicial sets. The category of simplicially enriched operads
will be denoted $\sSetOp$. For later use we say that a morphism between monochromatic topological
operads is a \emph{weak equivalence} if it is a levelwise weak homotopy equivalence.
\\
To compare the theories of topological operads and simplicially enriched operads we use a fact similar to
\cite[Cor. 1.14]{BM13} that a Quillen equivalence
$\textsf{V} \rightarrow \textsf{V'}$ between suitably nice symmetric monoidal model categories induces a Quillen equivalence
$\textsf{V-Op} \rightarrow \textsf{V'-Op}$ between the model structures on
enriched operads. A (symmetric) monoidal category $\textsf{C}$ equipped with a model structure is called a \emph{(symmetric) monoidal model category} if it
satisfies the following two axioms.
\begin{itemize}
 \item For every pair of cofibrations $f \colon X \rightarrow Y$, $f' \colon X' \rightarrow Y'$ the map
 \[
  (X \otimes Y') \coprod_{(X \otimes X')} (Y \otimes X') \rightarrow Y \otimes Y'
 \]
is a cofibration. It is a weak equivalence if $f$ or $f'$ is.
\item For every cofibrant $X$ the morphism
\[
 QI \otimes X \rightarrow I \otimes X \rightarrow X
\]
is a weak equivalence. Here $QI \rightarrow I$ denotes a cofibrant replacement of the tensor unit $I$.
\end{itemize}

These axioms are called the \emph{pushout-product axiom} and the \emph{unit axiom}, respectively. Examples include the usual model categories of simplicial
sets, compactly generated weak Hausdorff spaces and chain complexes.
\begin{exam}
 Let $(\textsf{C},\otimes)$ be a closed symmetric monoidal category (a monoidal category is \emph{closed} if the tensor product has a right adjoint,
 the \emph{internal Hom}) and $X$ an object of $\textsf{C}$.
 The \emph{endomorphism operad} $\textup{End}(X)$ is the operad enriched in $\textsf{C}$ on one object with morphism objects
\[
\textup{End}(X)(n) = \underline{\Hom}_{\textsf{C}}(X^{\otimes n}, X)
\]
and the obvious multicomposition by insertion. The functor $\underline{\Hom}$ denotes the internal $\Hom$-functor of $\textsf{C}$.
\end{exam}

Endomorphism operads give a way for other operads to act on objects of $\textsf{C}$. In this way operads classify additional algebraic structures.

\begin{defi}
 An \emph{algebra} $A$ over a monochromatic $\textsf{C}$-operad $P$ is an object $A$ of $\textsf{C}$ together with a map of operads
$P \to \textup{End}(A)$.
\end{defi}

\begin{exam}
 Let $\operatorname{Com}$ be the terminal topological operad. It has a single object and every mapping space
 is a point. Let $X$ be a topological space. Then any map $f$ from
$\operatorname{Com}$ to $\operatorname{End}(X)$
turns $X$ into an abelian topological monoid with operation
$f(\ast_2) \in \Map(X \times X,X)$.
\end{exam}

We will now describe the operads central to this work. The little disk operads have been studied in great detail. In \cite{May72} May proved his famous
recognition principle that a connected space is an algebra over the little $n$-disk operad if and
only if it is weakly equivalent to an $n$-fold loop space. A more precise statement will be given after we defined these operads.

\begin{exam}[The little disks operad] \label{littlediskdef}
 Let $D_n(k)$ denote the topological space of disjoint, rectilinear (i.e. respecting parallel lines) embeddings $\coprod_kI^n \rightarrow I^n$. Composition
of (multi-)morphisms is given by identifying the image of one morphism with a $I^n$ in the domain of the next one.
The following image shows a composition map $D_2(1) \times D_2(2) \rightarrow D_2(2)$:
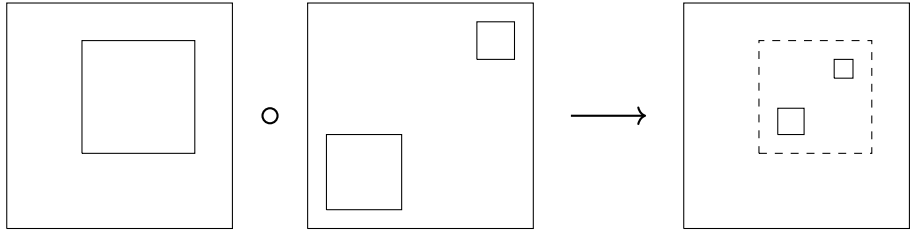
\[
\begin{tikzpicture}
\draw (0,0) -- (3,0) -- (3,3) -- (0,3) -- (0,0);
\draw (4.25,0.25) -- (5.25,0.25) -- (5.25,1.25) -- (4.25,1.25) -- (4.25,0.25);
\draw (6.25,2.25) -- (6.75,2.25) -- (6.75,2.75) -- (6.25,2.75) -- (6.25,2.25);
\draw (4,0) -- (7,0) -- (7,3) -- (4,3) -- (4,0);
\draw (1,1) -- (2.5,1) -- (2.5,2.5) -- (1,2.5) -- (1,1);
\draw (9,0) -- (12,0) -- (12,3) -- (9,3) -- (9,0);
\draw[dashed] (10,1) -- (11.5,1) -- (11.5,2.5) -- (10,2.5) -- (10,1);
\draw (10.25,1.25) -- (10.6,1.25) -- (10.6,1.6) -- (10.25,1.6) -- (10.25,1.25);
\draw (11,2) -- (11.25,2) -- (11.25,2.25) -- (11,2.25) -- (11,2);
\draw[thick] (3.5,1.5) circle (0.1cm);
\draw[thick,->] (7.5,1.5) -- (8.5,1.5) node[anchor=north west] {};
\end{tikzpicture}
\]
Any topological operad weakly equivalent to the operad of litte $n$-disks is called an $E_n$-operad.
\end{exam}
We describe another model of topological $E_n$-operads called Fulton-MacPherson operads.
This one is less intuitive but has properties more closely related
to objects we will investigate later on. It is built from a sequence of compactified euclidean configuration spaces. This construction
is due to Fulton-MacPherson \cite{FM94} as an algebraic compactification of complex varieties
and was later built in a topological way
by Axelrod-Singer and Sinha \cite{AS94,Sin04}.

\begin{exam}[The Fulton-MacPherson operad] \cite{Sin04,GeJo} \label{FMop}
 For every $n$ and $k$ the subgroup $G_{n}$ of $\operatorname{Aff}(n)$, the group of affine automorphisms, generated by translations and scalar multiplication,
 of $\RR^n$ acts freely on the ordered configuration space $\Conf(k,\RR^n)$. The quotient
 $C[k,n]$ is a manifold of dimension $n(k-1)-1$ with an induced $\Sigma_k$-action.
 Consider the collection (or symmetric sequence) $\mathsf F_n(k)$ given by these manifolds for $k \geqslant 2$ and set
 $\mathsf F_n(0) = \mathsf F_n(1) = \emptyset$. The \emph{Fulton-MacPherson $E_n$-operad} $FM_n$ has the same underlying set as the
 free operad $\Free(\mathsf F_n)$ together with a point in degree zero. (The definition of symmetric sequences and the free operad construction will be
 given in \ref{freeoperaddefi}.)
 Its topology is constructed in such a way that every level $FM_n(k)$ is a compact, connected manifold with corners. The interior of this manifold
 is $F_n(k)$, provided $k \geq 2$. The spaces $FM_n(0)$ and $FM_n(1)$ are one-point spaces.\\
 We will now give an explicit construction. For all $(i,j) \in \binom{k}{2}$ define the maps
 \begin{align*}
 a_{(i,j)} \colon \Conf(k,\RR^n) & \rightarrow S^{n-1} \\
   x & \mapsto \frac{x_i - x_j}{\lVert x_i - x_j \rVert}
 \end{align*}
 Furthermore for all $(i,j,k) \in \binom{k}{3}$ define the maps
 \begin{align*}
  b_{(i,j,k)} \colon \Conf(k,\RR^n) & \rightarrow [0,\infty] \\
    x & \mapsto \frac {\lVert x_i-x_j \rVert}{\lVert x_i-x_k \rVert}.
 \end{align*}
 The configuration space
 $\Conf(k,\RR^n)$ embeds into $\RR^{nk}\times (S^{n-1})^{\binom{k}{2}} \times [0,\infty]^{\binom{k}{3}}$ via
 \[
  x \mapsto (x, \prod_{i,j} a_{i,j}(x), \prod_{i,j,k}b_{i,j,k}(x)).
 \]
The closure of the image of this map shall be denoted $C_k[\RR^n]$. The action of $G_n$ on the configuration space extends to an action on
$C_k[\RR^n]$. The quotients of this action assemble to the operad $FM_n$, i.e. $FM_n(k) := C_k[\RR^n] / G_n$. These quotient spaces are compact manifolds
with corners. They have a natural
stratification we can use to understand the operad structure on the collection $FM_n$. \\
The stratification is indexed over the category $\Psi_k$ of rooted, labeled trees with $k$ leaves (non-root outer edges) and no vertices of valence
one or two. The set of leaves shall be labeled by the set $\{ 1, 2, \ldots, k \}$.
The morphisms in $\Psi_k$ are given by contraction of inner edges. So there is a map $S \rightarrow T$ if $S$ can be turned into $T$ by a sequence of
contractions of inner edges.
The stratum corresponding to a
tree $T$ with vertices $v_1, \ldots, v_l$ of valence $k_1, \ldots, k_l$ is homeomorphic to $\prod C[k_i-1,n]$. Its closure is the union of all the strata
indexed by trees mapping to $T$. In particular the interior of $FM_n(k)$ is diffeomorphic to $C[k,n]$. \\
We try to illustrate this with some examples. The stratifications of the first three spaces ($FM_n(0)$, $FM_n(1)$ and $FM_n(2)$) are trivial.
There is no non-corolla tree with fewer than three inputs and no vertex of valence one or two.
The first non-trivial stratification arises at level $3$. There are $4$ different trees in $\Psi_3$.
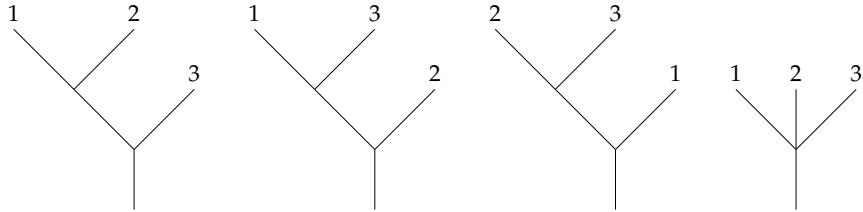
\[
\scalebox{0.8}{
 \begin{tikzpicture}

\draw (1,3) node[anchor=south]{1} --(2,2);
\draw (3,3) node[anchor=south]{2} --(2,2);
\draw (2,2) --(3,1);
\draw (4,2) node[anchor=south]{3} --(3,1);
\draw (3,1) --(3,0);

 \draw (5,3) node[anchor=south]{1} --(6,2);
\draw (7,3) node[anchor=south]{3} --(6,2);
\draw (6,2) --(7,1);
\draw (8,2) node[anchor=south]{2} --(7,1);
\draw (7,1) --(7,0);

\draw (9,3) node[anchor=south]{2} --(10,2);
\draw (11,3) node[anchor=south]{3} --(10,2);
\draw (10,2) --(11,1);
\draw (12,2) node[anchor=south]{1} --(11,1);
\draw (11,1) --(11,0);

\draw (13,2) node[anchor=south]{1} --(14,1);
\draw (14,2) node[anchor=south]{2} --(14,1);
\draw (15,2) node[anchor=south]{3} --(14,1);
\draw (14,1) --(14,0);
\end{tikzpicture}
}
\]
We see that there are $3$ strata homeomorphic to $S^{n-1} \times S^{n-1}$ and the corolla stratum corresponding to the interior of $FM_n(3)$.
(In the case $n = 1$ the configurations in $\mathbb{R}^n$ have a canonical ordering and by using this we obtain $FM_1(k) = \Sigma_k \times SP(k)$,
where $SP(k)$ is a polytope found by Stasheff long before the work of Fulton-MacPherson.)
The number of strata grows quickly with the level. There are already $26$ trees in $\Psi_4$. \\
This stratification is compatible with the operadic structure. So for example the composition
\[
FM_n(2) \times (FM_n(2) \times FM_n(2)) \rightarrow FM_n(4)
\]
is an embedding whose image is the union of the strata corresponding to trees of the shape
\[
 \begin{tikzpicture}
  \draw (1.5,3)  -- (2,2);
  \draw (2.5,3)  -- (2,2);
  \draw (2,2) -- (3,1);
  \draw (3,1) -- (3,0);
  \draw (4,2) -- (3,1);
  \draw (3.5,3)  -- (4,2);
  \draw (4.5,3)  -- (4,2);
 \end{tikzpicture}
\]
\end{exam}

\begin{thm}[May's recognition principle \cite{May72}\,]
Every $n$-fold loop space is an $E_n$-algebra in a canonical way. Conversely let $X$ be a group-like $E_n$-algebra.
Then there exists another space $Y$ and a zig-zag of weak homotopy equivalences
$X \leftarrow Z \rightarrow \Omega^n Y$ of $E_n$-algebras.
\end{thm}

The levelwise weak equivalences of simplicially enriched operads are the weak equivalences of a model structure on the category of monochromatic
operads.

\begin{thm} \emph{\cite[Thm 1.7]{CM11}} \label{sSetOpmodelstructure}
 The category $\sSetOp_*$ of monochromatic simplicially enriched operads carries a proper cofibrantly generated model structure
 such that the fibrations and weak equivalences are the levelwise fibrations and weak equivalences.
 \end{thm}
% sentence removed here ... because it was obscure ... it made sense in the dendroidal setting only.

\begin{defi} \label{freeoperaddefi}
Let $\textsf{C}$ be a symmetric monoidal category. A \emph{collection} (also known as \emph{symmetric sequence}) in $\textsf{C}$ is a sequence of objects $X_n$ (where $n \geq 0$)
with actions of the
 symmetric group $\Sigma_n$. More formally the category of collections in $\textsf{C}$ is the product of functor categories
 \[
  \textsf{Coll}(\textsf{C}) := \prod_{n \in \mathbb{N}} \textsf{C}^{\Sigma_n}
 \]
where the groups are regarded as groupoids with one object. The forgetful functor % name concealed: $\operatorname{UColl}$
taking a monochromatic $\textsf{C}$-operad to its underlying
collection has a left adjoint, called the \emph{free operad} functor. It is described at length in \cite[Chapter 5.8]{BM03}.
In each level $k$ a free operad is indexed by rooted trees with $k$ leaves.
Let $\mathbb{T}$ be the groupoid of finite rooted trees and isomorphisms.
More precisely $\mathbb T$ is the maximal subgroupoid of the dendrex category $\Omega$ defined in section~\ref{subsec-dendrex}.
Similarly let $\mathbb{T}_{\Lambda}$ be the groupoid of finite, rooted trees together with a total order $\lambda$ on their set of leaves.
For every collection $X$ we can define a functor
\[
 \underline{X} \colon \mathbb{T}^{\op} \rightarrow \textsf{C}
\]
by setting $\underline{X}(\eta) = I$, the tensor unit of $\textsf{C}$. Every tree $T \in \mathbb{T}$ can inductively be written as a grafting
$\corl_n \circ (T_1, \ldots, T_n)$. (The tree $\corl_n$ is the $n$-corolla, the tree with a single vertex of valence $n+1$.
These corollas will be introduced in \ref{deficorolla}.)
We set
\[
 \underline{X}(T) := X(n) \otimes \underline{X}(T_1) \otimes \ldots \otimes \underline{X}(T_n).
\]
The \emph{free operad} on a collection $X$ has the $n$-th space
\[
 \free(X)(n) \cong \coprod_{\substack{[(T,\lambda)] \in \pi_0 \mathbb T_{\Lambda} \\ T \text { has n inputs} }} \underline{X}(T) / \Aut(T,\lambda).
\]
Note that objects of $\mathbb T_{\Lambda}$ can have non-trivial automorphisms.
There is an understanding that we choose a representative $(T,\lambda)$ in each element of $\pi_0 \mathbb{T}_{\Lambda}$.
The action of $\Sigma_n$ on $\free(X)(n)$ comes from the action of $\Sigma_n$ on the total orderings of the leaves.
The permutation $\sigma \in \Sigma_n$
sends $(T,\lambda)$ to the chosen representative $(T',\lambda')$ in the class of $(T,\sigma(\lambda))$. We need to choose an
isomorphism from $(T,\sigma(\lambda))$
to the representative $(T',\lambda')$ in order to get an isomorphism from $\underline{X}(T)$ to $\underline{X}(T')$. Consequently that isomorphism
is well defined only modulo the action of $\Aut(T,\lambda)$ on $\underline{X}(T)$. ---
The operadic composition in a free operad is induced by grafting of trees in the obvious way.
\end{defi}

\begin{rema} \label{remainduced} \cite[Thm. 5.1]{GJ10}
 Let
 \[
  F \colon \textsf{C} \leftrightarrows \textsf{D} \colon G
 \]
be an adjunction of categories. Let $(\mathcal{C}o, \mathcal W, \mathcal Fi)$ be a cofibrantly generated model structure on $\textsf{C}$. A morphism
$f \colon a \rightarrow b$ in $\textsf{D}$ shall be called a fibration or weak equivalence if its image under $G$ is. If
\begin{itemize}
 \item $G$ preserves filtered colimits
 \item every morphism of $\textsf{D}$ with the left lifting property with respect to all fibrations is a weak equivalence
\end{itemize}
then there exists a cofibrantly generated model structure on $\textsf{D}$ with the above fibrations and weak equivalences.
Furthermore if $I$ is the set of generating cofibrations of $\textsf{C}$ and $J$ the set of generating trivial cofibrations then $F(I)$ and $F(J)$ are the
sets of generating cofibrations and trivial cofibrations, respectively, of $\textsf{D}$.
This model structure is called the (left) \emph{transferred model structure} along the adjunction $(F \dashv G)$.
\end{rema}

Since the model structure of \ref{sSetOpmodelstructure} is transferred from the category of collections we immediately see that any free operad on a cofibrant
collection is cofibrant.

A functorial cofibrant replacement of monochromatic topological operads has been constructed by Boardman and Vogt in \cite{BV73} and generalized
by Berger and Moerdijk to the
case of operads enriched in suitable model categories in \cite{BM06} and to the multi-object case in \cite{BM07}.
To avoid unnecessary complexity only the version for monochromatic topological operads will be presented here.

\begin{defi}[The Boardman-Vogt construction] \label{BVconst}
 Let $P$ be a topological operad. The \emph{Boardman-Vogt $W$ construction} is a factorization
 \[
  \operatorname{free}(P) \hookrightarrow WP \xrightarrow{\sim} P
 \]
of the counit $\operatorname{free}(P) \rightarrow P$ into a cofibration followed by a weak equivalence. The operad $WP$ itself is also often called the
Boardman-Vogt $W$-construction. Under a small hypothesis on $P$ the BV construction $WP$ is a functorial cofibrant replacement of $P$. \\
To build this factorization we start with the free operad $\operatorname{free}(P)$. Recall that its $n$-ary operations are the labellings of
certain trees with $n$ leaves. Each vertex in these trees (of valence $k+1$) is colored by an element of $P(k)$. To get $WP$ we furthermore equip the
internal edges with a length $l_e \in [0,1]$. If some edge $e$ has length $0$ then this point in $WP$ is identified with the one given by contracting the
edge and composing the two adjacent operations.
\end{defi}

\begin{lem} \emph{\cite{BV73,BM06}}
 If the underlying collection of $P$ is $\Sigma$-cofibrant
 \emph{(every space of the collection is cofibrant and the action of $\Sigma_k$ on the $k$-th space of the collection is free for all $k$)}, then the operad $WP$ is cofibrant.
\end{lem}

\subsection{Dendroidal sets and spaces} \label{subsec-dendrex}
The concept of dendroidal set was introduced by Ieke Moerdijk and Ittay Weiss in \cite{MW07} as
a generalization of simplicial sets suited to describe and investigate the homotopy theory of operads.
The homotopy theory of dendroidal sets and spaces was developed by Cisinski and Moerdijk in a series of papers \cite{CM09,CM11,CM13}.
Dendroidal sets correspond to (higher) operads in exactly the same way simplicial sets do to (higher) categories. Many constructions
for simplicial objects have dendroidal analogues. A major tool in this article will be the notion of dendroidal complete Segal spaces. \\
In this section we will give a short introduction to these notions and quote the most important results for our further work.

\begin{defi}
 A \emph{tree} (or \emph{dendrex}) $T$ consists of a tuple $(T,\leq,L)$ such that $(T,\leq)$ is a partially ordered finite set
 (the set of edges) with a unique minimal element (called the \emph{root}) and the
 property that for each element $x\in T$ the set of elements smaller than $x$ is linearly ordered.
 The set $L$ is a subset of the set of maximal elements of $T$. The elements of $T$ are called \emph{edges} and the elements of $L$ are
 called \emph{leaves}. An edge is \emph{inner} if it is neither a leaf nor the root. For any edge $x \in T \smin L$ the set $\inco(x)$ of elements $y > x$ such that
 there is no $z$ with $y > z > x$ is called the set of \emph{incoming edges} (or \emph{inputs}) of $x$. For any $x \in T \smin L$ the set
 $v_x :=\{x\} \cup \inco(x)$
 is a \emph{vertex} of $T$. (The set of vertices is in obvious bijection to $T\smin L$.)
 \end{defi}

These trees can be arranged into a category $\Omega$. To define the morphisms of $\Omega$ we note that every tree $T$
determines an operad $\Omega(T)$ whose set of objects is the set of edges of $T$. Every vertex $v_x$ contributes a generating
operation whose input set is the set of incoming edges $\inco(x)$ and whose
output is $x$. For example the
operad generated by
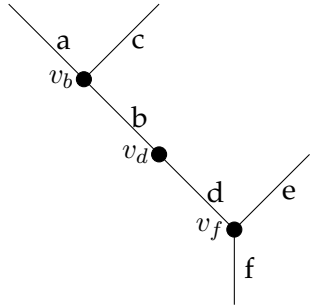
\[
\begin{tikzpicture}
\draw (0,4) --(1,3) node[midway, left] {a};
\draw (1,3) --(2,2) node[midway, left] {b};
\draw (2,4) --(1,3) node[midway, right] {c};
\draw (2,2) --(3,1) node[midway, left] {d};
\draw (4,2) --(3,1) node[midway, right] {e};
\draw (3,1) --(3,0) node[midway, left] {f};
\draw[fill] (2,2) circle [radius=0.1] node[left] {$v_d$};
\draw[fill] (1,3) circle [radius=0.1] node[left] {$v_b$};
\draw[fill] (3,1) circle [radius=0.1] node[left] {$v_f$};
\end{tikzpicture}
\]
has $6$ objects, morphisms $v_d \in \Omega(T)(b;d)$, $v_b \in \Omega(T)(a,c;b)$, $v_f \in \Omega(T)(d,e;f)$ and their compositions
$v_d \circ v_b \in \Omega(T)(a,c;d)$, $v_f \circ v_d \in \Omega(T)(b,e;f)$ and $v_f \circ v_d \circ v_b \in \Omega(T)(a,c,e;f)$.
The set of morphisms in $\Omega$ between two trees is defined to be set of morphisms between their corresponding operads.
\[
 \Hom_{\Omega}(T,T') := \Hom_{\textsf{Op}}(\Omega(T),\Omega(T')).
\]
Note that the morphisms do not have to preserve the root.

\begin{exam} \label{deficorolla}
 The trees with exactly one vertex are of particular importance and are called \emph{corollas}.
The following figure shows the $3$-corolla, the $1$-corolla and the $0$-corolla:
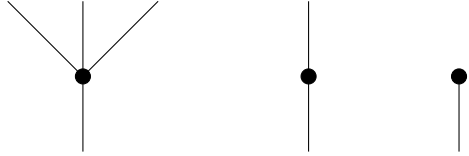
\[
\begin{tikzpicture}
\draw (0,2) --(1,1);
\draw (1,2) --(1,1);
\draw (2,2) --(1,1);
\draw (1,1) --(1,0);
\draw[fill] (1,1) circle [radius=0.1];
\draw (4,1) --(4,0);
\draw (4,2) --(4,1);
\draw[fill] (4,1) circle [radius=0.1];

\draw (6,1) --(6,0);
\draw[fill] (6,1) circle [radius=0.1];

\end{tikzpicture}
\]
\end{exam}

A presheaf on $\Omega$ is called a \emph{dendroidal set}. More generally for any
symmetric monoidal category $\textsf{C}$ the objects of $\Fun(\Omega^{\op}, \textsf{C})$ are called \emph{dendroidal objects} in $\textsf{C}$.
The category of dendroidal objects in $\textsf{C}$ will be denoted by $d\textsf{C}$. For every object $T$ in $\Omega$ there is the dendroidal
set represented by $T$; it is denoted by $\Omega[T]$.  \\
The simplex category $\Delta$ embeds into $\Omega$ as a full subcategory by sending $[n]$ to the linear tree with $n$ vertices and $n+1$ edges.
The operads $\Omega(T)$ for $T$ in the image of this embedding have no morphisms of higher degree and are thus equivalent to categories.
They are easily seen to be the linear categories $[n]$. There is a tree $\eta$ in $\Omega$ with exactly one edge; it is also
the image of $[0]$ in $\Delta$. Every operad which admits a morphism to $\Omega(\eta)$ cannot have higher morphisms
and thus $\textsf{Op}/\Omega(\eta) = \textsf{Cat}$ and $\Omega/\eta \cong \Delta$ and $\dSet/\Omega[\eta] = \sSet$. \\
Several constructions on the category of simplicial sets can be generalized to the dendroidal setting and recovered by
the description of $\sSet$ as the overcategory $\dSet/\Omega[\eta]$. One of the most important is the nerve construction. For an operad $P$ the \emph{dendroidal nerve}
$N_dP$ is the dendroidal set given by
\[
 N_dP(T) = \Hom_{\textsf{Op}}(\Omega(T),P).
\]
The nerve functor has a left adjoint $\tau_d$. It can be described as the unique colimit preserving functor that sends the represented
presheaf $\Omega[T]$ to the operad $\Omega(T)$. \\
For any category regarded as an operad
the dendroidal nerve reduces to the ordinary nerve of a category. Hence the following square of functors
\[
\begin{tikzcd}
 \textsf{Cat} \arrow[r] \arrow[d] & \sSet \arrow[d] \\
 \textsf{Op} \arrow[r] & \dSet
\end{tikzcd}
\]
commutes.

\begin{rema}
For every $T$ in $\Omega$ there is an isomorphism of dendroidal sets
\[ \Omega[T]\cong N_d\Omega(T).\]
\end{rema}

We will now examine the category
$\Omega$ more closely and describe the homotopy theory of dendroidal objects.

\begin{defi}
  Morphisms (in $\Omega$) of the following kind are called \emph{inner face maps}:
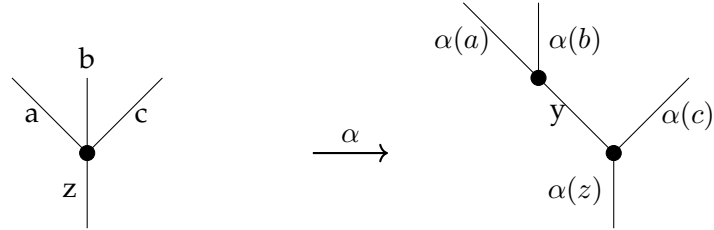
\[
\begin{tikzpicture}
\draw (0,2) --(1,1) node[midway,left] {a};
\draw (1,2) --(1,1) node[at start,above] {b};
\draw (2,2) --(1,1) node[midway,right] {c};
\draw (1,1) --(1,0) node[midway,left] {z};
\draw[fill] (1,1) circle [radius=0.1];
\draw [thick, ->] (4,1) -- (5,1) node[midway,above] {$\alpha$};
\draw (6,3) --(7,2) node[midway,left] {$\alpha(a)$};
\draw (7,3) --(7,2) node[midway,right] {$\alpha(b)$};
\draw (9,2) --(8,1) node[midway,right] {$\alpha(c)$};
\draw (7,2) --(8,1) node[midway,left] {y};
\draw (8,1) --(8,0) node[midway,left] {$\alpha(z)$};
\draw[fill] (7,2) circle [radius=0.1];
\draw[fill] (8,1) circle [radius=0.1];
\end{tikzpicture}
\]
They (contravariantly) correspond to operadic composition. Morphisms of the following kind
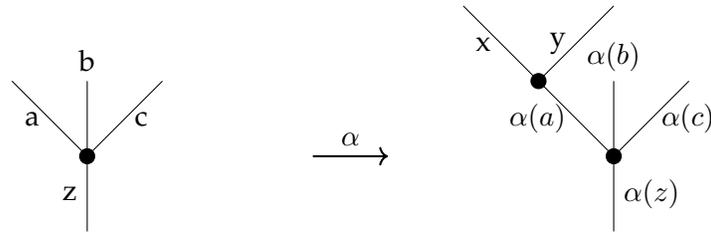
\[
\begin{tikzpicture}
\draw (0,2) --(1,1) node[midway,left] {a};
\draw (1,2) --(1,1) node[at start,above] {b};
\draw (2,2) --(1,1) node[midway,right] {c};
\draw (1,1) --(1,0) node[midway,left] {z};
\draw[fill] (1,1) circle [radius=0.1];
\draw [thick, ->] (4,1) -- (5,1) node[midway,above] {$\alpha$};
\draw (6,3) --(7,2) node[midway,left] {x};
\draw (8,3) --(7,2) node[midway,left] {y};
\draw (7,2) --(8,1) node[midway,left] {$\alpha(a)$};
\draw (8,2) --(8,1) node[at start,above] {$\alpha(b)$};
\draw (9,2) --(8,1) node[midway,right] {$\alpha(c)$};
\draw (8,1) --(8,0) node[midway,right] {$\alpha(z)$};
\draw[fill] (8,1) circle [radius=0.1];
\draw[fill] (7,2) circle [radius=0.1];
\end{tikzpicture}
\]
are called \emph{outer face maps}.
The \emph{degeneracies} are morphisms given by deleting an inner vertex of valence 2.
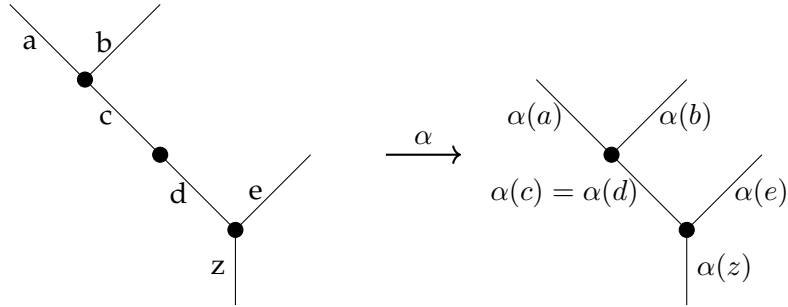
\[
\begin{tikzpicture}
\draw (0,4) --(1,3) node[midway,left] {a};
\draw (1,3) --(2,2) node[midway,left] {c};
\draw (2,4) --(1,3) node[midway,left] {b};
\draw (2,2) --(3,1) node[midway,left] {d};
\draw (4,2) --(3,1) node[midway,left] {e};
\draw (3,1) --(3,0) node[midway,left] {z};
\draw[fill] (2,2) circle [radius=0.1];
\draw[fill] (1,3) circle [radius=0.1];
\draw[fill] (3,1) circle [radius=0.1];
\draw [thick, ->] (5,2) -- (6,2) node[midway,above] {$\alpha$};
\draw (7,3) --(8,2) node[midway,left] {$\alpha(a)$};
\draw (9,3) --(8,2) node[midway,right] {$\alpha(b)$};
\draw (8,2) --(9,1) node[midway,left] {$\alpha(c) = \alpha(d)$};
\draw (10,2) --(9,1) node[midway,right] {$\alpha(e)$};
\draw (9,1) --(9,0) node[midway,right] {$\alpha(z)$};
\draw[fill] (8,2) circle [radius=0.1];
\draw[fill] (9,1) circle [radius=0.1];
\end{tikzpicture}
\]
On the operads associated to these trees this induces the map identifying the two adjacent objects and sending
the unary morphism between them to the identity on the new object.
By \cite[Lemma 3.1]{MW07} every morphism in $\Omega$ factors up to isomorphism as a composition of degeneracies followed by a sequence of face maps.
\end{defi}

The theory of \emph{Segal spaces} was developed by Rezk in \cite{Rez01}.
These Segal spaces are simplicial spaces behaving like an up-to-homotopy version of the nerve of a topological category. \\
Its dendroidal generalization was constructed in \cite{CM13}. This model has the merit of being less rigid than enriched operads
in a sense that their composition
law is only defined up to a contractible choice. \\
More exactly our model will be based on simplicial dendroidal sets.
As a category of simplicial presheaves it is canonically tensored, cotensored and enriched
over simplicial sets. The tensoring is given by taking a dendrexwise product of simplicial sets.

\begin{defi}
 Let $X$ and $Y$ be dendroidal spaces. A morphism $f \colon X \rightarrow Y$ is called a \emph{weak equivalence} if
 for every tree $T$ in $\Omega$ the map $X_T \rightarrow Y_T$ is a weak equivalences of simplicial sets.
\end{defi}

There are three standard choices for the classes of fibrations and cofibrations on the category of simplicial dendroidal sets if we fix the
class of dendrexwise weak equivalences as our choice for the weak equivalences. The \emph{projective} model structure is uniquely determined
by defining a morphism to be a fibration if and only if it is a dendrexwise Kan fibration. Dually the \emph{injective} model structure
is uniquely determined by the choice of dendrexwise cofibrations as its class of cofibrations. \\
There is an intermediate model structure taking into account the Reedy structure of $\Omega$.
Theorem \ref{dSpaceReedy} describes this in more detail. This model structure is a
central starting point in \cite{CM13}. Since we are flexible in our choice of model structure we will not need to use this result.

\begin{thm} \emph{\cite[Prop. 5.2]{CM11}} \label{dSpaceReedy}
 The category $\sdSet$ of simplicial dendroidal sets can be equipped with a
 generalized Reedy model structure using the Reedy structure of $\Omega$.
 It is cofibrantly generated and proper. The weak equivalences
 are the dendrexwise simplicial weak equivalences.
 A map of simplicial dendroidal sets $X \rightarrow Y$ is a
 fibration, resp. trivial fibration,
 if the relative matching maps
 \[
  X^{\Omega[T]} \rightarrow X^{\partial \Omega[T]} \times_{Y^{\partial \Omega[T]}} Y^{\Omega[T]}
 \]
are fibrations, resp. trivial fibrations, for all $T$. \emph{(See \cite[\S2.1]{CM11} for the meaning of $\partial\Omega[T]$.)}
 \end{thm}

\begin{defi} \label{SegalCoreDefi}
 Let $T \in \Omega$ be a tree. If $T$ has at least one vertex
 the \emph{spine} or \emph{Segal core} $\Sc[T]$ of $T$ is defined as a dendroidal subset of $\Omega[T]$
given by the union of all $\Omega[S]$ for subcorollas $S$ of $T$. (There is one subcorolla for each vertex of $T$.)
For the trivial tree $\eta$ without vertices
 we set $\Sc[\eta] = \Omega[\eta]$.
Note that we recover the definition of a spine of a simplex by applying this definition to linear trees.
\end{defi}

These Segal cores have a close connection to operads. Remember that a simplicial set $X$ is the nerve of a category if and only if all maps
\[
 X_n \rightarrow X_1 \times_{X_0}X_1\times_{X_0}  \ldots \times_{X_0} X_1
\]
induced by the spine inclusions are bijections.
The following lemma is the generalization of this fact to the dendroidal setting.

\begin{lem} \label{dendspine}
 A dendroidal set $X$ is the nerve of an operad if and only if the map
 \[
  \Hom_{\dSet}(\Omega[T],X) \rightarrow \Hom_{\dSet}(\Sc[T],X)
 \]
induced by the Segal core inclusion is a bijection for all trees $T$. \\
Similarly a dendroidal space $X$ is the nerve of a simplicially enriched operad if and only if the map of simplicial sets
 \[
  X^{\Omega[T]} \rightarrow X^{\Sc[T]}
 \]
is an isomorphism for all $T$.
\end{lem}

For dendroidal spaces to model topological operads we still want this equivalence to hold up to homotopy.
The resulting notion will extend the classical definition of a
complete Segal space as a model for $(\infty,1)$-categories.

\begin{defi} \label{dendsegaldefi}
 A dendroidal space $X$ is called a \emph{dendroidal Segal space} if for all trees $T$ the map
 \[
  X_T = \Hom(\Omega[T],X) = X^{\Omega[T]} \longrightarrow \RHom(\Sc[T],X)
 \]
is a weak equivalence of simplicial sets.
\end{defi}

\begin{rema}
  In \cite{CM13} Cisinski-Moerdijk define the \emph{model structure for dendroidal Segal spaces} as the left Bousfield localization
of the generalized Reedy structure on $\sdSet$ at the set of Segal core inclusions.
\end{rema}

(The definition \ref{dendsegaldefi} is not in full agreement with \cite{CM13} because Cisinski-Moerdijk write
$\Hom(\Sc[T],X)$ instead of $\RHom(\Sc[T],X)$ and insist that dendroidal Segal spaces
be Reedy-fibrant to make up for that. Namely, the Segal core $\Sc[T]$ is Reedy-cofibrant. Therefore $\Hom(\Sc[T],X)$ is weakly equivalent
to $\RHom(\Sc[T],X)$ if $X$ is Reedy fibrant.)

\begin{lem} \emph{\cite[Thm 7.8]{BW15}; \cite[Thm 4.3]{BHR19}}
 Let $P$ be a monochromatic simplicial operad. The dendroidal space $N_dP$
given by
\[
 (N_dP)_T := \underline{P}(T),
\]
using the notation of \emph{\ref{freeoperaddefi}}, satisfies the Segal property.
The assignment $P \mapsto N_dP$ is functorial and preserves all weak equivalences.
Moreover for any two operads $P$ and $Q$ the morphism
\[
 \RHom(P,Q) \rightarrow \RHom(N_dP,N_dQ)
\]
is a weak equivalence.
\end{lem}

\begin{rema} In both \cite{BW15} and \cite{BHR19}, this result is attributed to Cisinski and Moerdijk, but 
it is not stated exactly in this form by Cisinski and Moerdijk. A small adjustment is required and \cite{BHR19} explain this in detail. 
Throughout this article we will only use the statement for $1$-reduced operads.
This implies that $N_dP$ is complete (and Segal). In general $N_dP$ is not complete.
\end{rema}

\section{A tower of derived mapping spaces} \label{sec-tower}

\subsection{Construction of the tower} \label{subsec-towconst}

In section~\ref{sec-intro} we have introduced the notion of dendroidal Segal spaces as a model for the homotopy theory of topological operads. We want to
use this model to describe the derived mapping spaces between two topological operads. For any two objects $X$ and $Y$ in any model category $\textsf{C}$
this derived mapping
space can be defined as the space of maps $\Hom_\textsf{C}(X^{c},Y^f)$ from a cofibrant replacement of $X$ to a fibrant
replacement of $Y$. Although it is a slick definition, actual computations can be cumbersome
because these objects tend to be unwieldy. Moreover, although this definition inherently depends on the choice of a model structure, the homotopy
type of the derived mapping space only depends on the class of weak equivalences. There is another more general definition of a derived mapping space
due to the work \cite{DK80b} of Dwyer and Kan. For every category $\textsf{C}$ together with a subcategory $W$ of weak equivalences they define derived mapping
spaces in terms of zig-zags of morphisms. If $W$ happens to be the class of weak equivalences of a model structure on $\textsf{C}$, then both
definitions yield weakly equivalent derived mapping spaces. \\
We start this section with a general investigation of derived mapping spaces in categories of space-valued functors with levelwise weak equivalences
under the assumption that the indexing category $\textsf{C}$ can be written as a sequential colimit of full subcategories $\textsf{C}_i$. We prove a lemma that the
derived space of natural transformations in $\Fun(\textsf{C}^{\op},\sSet)$ can be recovered up to homotopy from the mapping spaces between the
restrictions of these functors to the subcategories $\textsf{C}_i$.

\begin{lem} \label{TowerConvergence}
Let $F$ and $G$ be
contravariant functors from $\textup{\textsf{C}}$ to $\sSet$. We call a natural transformation $F \rightarrow G$ a
weak equivalence if it is an objectwise weak equivalence of simplicial sets in the sense of Kan-Quillen. Let $U_i$ denote the restriction functor from
$\Fun(\textup{\textsf{C}}^{\op},sSet)$ to $\Fun(\textup{\textsf{C}}_i^{\op},\sSet)$. Then the natural morphism
\[
 \RHom(F,G) \rightarrow \holim_i \RHom(U_iF,U_iG)
\]
is a weak equivalence.
\end{lem}

We will prove this lemma in two steps. First we show that every contravariant functor admits a weak equivalence from a
functor satisfying a cellularity property.
These \emph{free CW-functors} are a subclass of the CW-functors of Dror-Farjoun \cite[1.16]{DF}. We then prove the statement
\ref{TowerConvergence} for all free CW-functors $F$.

\begin{defi}
 Let $\textsf{C}$ be a category.
 A functor $F \colon \textsf{C}^{\op} \rightarrow \sSet$ is called a \emph{free CW-functor} if there is a sequence
 \[
 \emptyset = F_{-1} \subset F_0 \subset F_1 \subset \ldots \subset F_{i-1} \subset F_i \subset \ldots
 \]
 of subfunctors of $F$ such that the
 following properties are satisfied.
 \begin{enumerate}
  \item $F(x) = \colim_i F_i(x)$ for all objects $x$ of $\textsf{C}$.
  \item For all $i\ge 0$ there exists a pushout diagram
  \[
   \begin{tikzcd}
   K_i \times \partial \Delta[i] \arrow[r] \arrow[d] & F_{i-1} \arrow[d] \\
    K_i \times \Delta[i] \arrow[r] & F_i
   \end{tikzcd}
  \]
  where $K_i$ is a disjoint union of representable functors.
 \end{enumerate}
\end{defi}

\begin{exam}
 Let $\underline{G}$ be a group regarded as a category with one object. Then
free CW-functor $F$ is nothing but a simplicial set $F(*)$ with a free $G$-action.
The subfunctors $F_i$ of $F$ can be the skeletons of $F(*)$.
\end{exam}

\begin{lem}
 For every functor $G \colon \textup{\textsf{C}}^{\op} \rightarrow \sSet$
 there is a free CW-functor $F$ together with a natural equivalence $F \rightarrow G$.
\end{lem}

\begin{proof} Fix some $n\ge 0$ and suppose, for an induction argument, that we have already constructed
a free CW-functor $D$ together with a natural transformation
 $u\colon D\rightarrow G$ such that
 for each $c$ in $\textsf{C}$ the morphism $u_c\co D(c) \rightarrow G(c)$
 is $(n-1)$-connected. We want to get rid of the relative homotopy groups
 $\pi_n(G(c),D(c))$ for all $c \in \textsf{C}$. (Strictly speaking we should write $\pi_n(Z(c),D(c))$ where $Z(c)$ is the mapping cylinder of
 $u_c$.) Let $x \in  \pi_n(G(c),D(c))$ be a non-trivial element of this homotopy group and let
 \[ v_{c,x} \colon (K_{c,x},L_{c,x}) \rightarrow G(c) \]
 denote a representative of $x$, where $(K_{c,x},L_{c,x})$ is a (possibly iterated) barycentric subdivision of the pair $(\Delta[n], \partial \Delta[n])$.
 Let $D_{c,x}$ be the pushout of
 \[
  \Hom(-,c) \times K_{c,x} \hookleftarrow \Hom(-,c) \times L_{c,x} \rightarrow F
 \]
(where the right-hand arrow extends, and is determined by, $v_{c,x}$ restricted to $L_{c,x}$).
Let $E$ be the union along the common subfunctor $D$ of the $D_{c,x}$ where $c$ ranges over all objects of $\textsf{C}$ and
$x$ ranges over all all non-trivial elements of the homotopy groups $\pi_n(G(c),E(c))$. The choices $v_{c,x}$ together with $u$ uniquely
define a new natural transformation $v\co E\rightarrow G$. By construction this specializes to an $n$-connected map
$E(c) \rightarrow G(c)$ for every $c$ in $\textsf{C}$. It remains to be shown that the functor $E$ is again a free CW-functor.
To do so we show that the pairs $(K_{c,x},L_{c,x})$ are pairs of cell complexes. Each non-degenerate simplex in $K_{c,x} \smin L_{c,x}$
contributes a free cell to $D_{c,x}$ which is not in $D$.
It follows that $D_{c,x}$ is free CW. Since different choices of $(c,x)$ lead to disjointly attached cells, the union $E$
is free CW as well (and what is more important, we have shown that it is free CW \emph{relative} to $D$).

The functor $F$ can now be constructed as the union (sequential colimit) of an increasing sequence
\[  F^{-1} \subset F^0 \subset F^1 \subset F^2 \subset \ldots \]
of functors $\textsf{C}\to \sSet$, each equipped with a morphism $w^n\co F^n\to G$, such that $w^n$ extends $w^{n-1}$.
Define $F^{-1}=\emptyset$ and define $F^{n}$ and $w^n$ inductively so that $F^n$ is to $F^{n-1}$ as $E$ is to $D$ above,
and $w^n$ is to $w^{n-1}$ as $v$ is to $u$. The union of the $w^n$
is a morphism $w\co F\to G$ and it is a weak equivalence by construction.
\end{proof}

\begin{prop} \label{principle}
 Suppose we have some property $\mathfrak{P}$ for contravariant functors
 from $\textsf{C}$ to $\sSet$. Assume the property $\mathfrak{P}$ is preserved under levelwise weak equivalence, disjoint unions over arbitrary indexing sets
 and homotopy pushouts and holds for all representable functors.
 Then $\mathfrak{P}$ holds for every contravariant functor $F$ from $\textsf{C}$ to $\sSet$.
\end{prop}

\begin{proof}
 Without loss of generality we assume $F$ to be a free CW-functor.
 We will prove this in two steps. First we show by induction that all skeleta $F_i$ have the desired property and then deduce the statement
 for the homotopy colimit.
 The $0$-skeleton is just a disjoint union of representable
 functors and as such has the property $\mathfrak{P}$. To prove the induction step we have to show that $F_{i+1}$ is a
 homotopy pushout of functors satisfying $\mathfrak{P}$.
 Note that because of the homotopy invariance of $\mathfrak{P}$ every cell $\Hom(-,c) \times \Delta[i]$ has property $\mathfrak{P}$.
 Since we assumed the property to be preserved under disjoint union the coproduct $\coprod \Hom(-,c) \times \Delta[i]$ still has
 the property. To conclude the induction step we need to show our desired property for every functor of the form
 $\Hom(-,c) \times \partial \Delta[i]$. But this is already covered by the induction assumption because $\Hom(-,c) \times \partial \Delta[i]$
 is a free CW-functor built from cells of dimension $(i-1)$ and less. --- Next, we have to
 show that the pushout diagram
\[
   \begin{tikzcd}
   K_i \times \partial \Delta[i] \arrow[r] \arrow[d] & F_{i-1} \arrow[d] \\
    K_i \times \Delta[i] \arrow[r] & F_i
   \end{tikzcd}
\]
is actually a homotopy pushout square.
But the left hand vertical morphism is levelwise injective and
 so the pushout square is levelwise a homotopy pushout.
 \\
 So far we have shown the property $\mathfrak{P}$ for all $k$-skeleta $F_k$. We now show that the homotopy colimit $F$ can
 be written as a homotopy pushout. The argument has already been presented by Milnor in \cite{Mil}.
 Let
 \[
 tF := F_0 \times [0,1] \cup F_1 \times [1,2] \cup F_2 \times [2,3] \ldots
 \]
 understood as a subfunctor of $F \times [0,\infty)$. (The intervals can be taken as copies of $\Delta[1]$.) This construction is also
 known as the \emph{telescope} associated to the skeletal filtration of $F$.
 Note that the inclusion of $tF$ into $F \times [0,\infty)$ is a weak equivalence and thus $tF$ has property $\mathfrak{P}$
 if and only if $F$ does. We want to show that $tF$ decomposes as a homotopy pushout of functors with property $\mathfrak{P}$.
 To do so we define the subfunctors $L_1$ and $L_2$ of $tF$ by setting
 \[
  L_1 := F_0 \times [0,1] \cup F_2 \times [2,3] \cup \ldots
 \]
and
\[
  L_2 := F_1 \times [1,2] \cup F_3 \times [3,4] \cup \ldots
 \]
as the even and odd parts of $tF$, respectively. Their intersection $L_1 \cap L_2$ is the functor
\[
 L_1 \cap L_2 \cong F_0 \times \{ 1 \} \cup F_1 \times \{ 2 \} \cup \ldots
\]
We can write $tF$ as the pushout of $L_1\leftarrow L_1\cap L_2\to L_2$.
The functors $L_i$ and $L_1\cap L_2$ are all weakly equivalent to disjoint unions of skeleta $F_j$ und thus have property $\mathfrak{P}$
by the previous discussion. The pushout is also a homotopy pushout because both maps $L_1 \cap L_2 \rightarrow L_i$ are cofibrations.
\end{proof}

Using this principle we can prove our lemma \ref{TowerConvergence}.

\begin{proof}[Proof of \ref{TowerConvergence}]
 We need to verify the three assumptions of the previous lemma. We fix a levelwise fibrant functor $G \in \Fun(\textsf{C}^{\op},\sSet)$
 throughout this investigation.\\
First assume $F$ to be representable by some object $c$ and let $\textsf{C}_k$ be the first subcategory of the sequence $\textsf{C}_{\bullet}$ to contain $c$. Because we assumed
all subcategories $\textsf{C}_i$ to be full subcategories the
restriction of $F$ to $\textsf{C}_k$ is isomorphic to the functor on $\textsf{C}_k$ represented by $c$. The same holds for all $\textsf{C}_n$ with $n \geqslant k$.
It follows that
\[
\RHom(U_nF,U_nG) \simeq U_nG(c) = G(c)
\]
for all $n \geqslant k$. It is immediate that the homotopy limit of the tower of derived mapping spaces is weakly equivalent to $\RHom(F,G)$.\\
Now assume $F$ is a disjoint union of functors $F_i$ for which the tower converges. Since disjoint unions are certainly preserved under restrictions we have
$U_nF = \coprod U_nF_i$. It follows that
\[
 \RHom(U_nF,U_nG) = \prod \RHom(U_nF_i,U_nG)
\]
for all $n$. We get a commutative square
\[
 \begin{tikzcd}
  \RHom(F,G) \arrow[r] \arrow[d] & \prod \RHom(F_i,G) \arrow[d] \\
  \holim \RHom(U_nF,U_nG) \arrow[r]  & \holim \prod \RHom(U_nF_i,U_nG).
 \end{tikzcd}
\]
The two horizontal morphism are weak equivalences by assumption, the right hand vertical morphism is a weak equivalence because we can commute the
homotopy limit with the product. It follows that the morphism
\[
 \RHom(F,G) \rightarrow \holim \RHom(U_nF,U_nG)
\]
is a weak equivalence.
\\
For the last step assume $F$ is the homotopy pushout of $F_1 \leftarrow F_0 \rightarrow F_2$ and
the tower converges for all $F_i$. We can arrange the derived mapping spaces in a commutative cube.
\[
\begin{tikzcd}[row sep={40,between origins}, column sep={80,between origins}]
    & \RHom(F,G)  \ar{rr} \ar{dd} \ar{dl} & & \ar{dl} \ar{dd} \RHom(F_2,G) \\
   \RHom(F_1,G) \ar[crossing over]{rr} \ar{dd} & & \RHom(F_0,G) \\
      & \holim_n \RHom(U_nF,U_nG)  \ar{rr} \ar{dl} & &  \holim_n \RHom(U_nF_2,U_nG)  \ar{dl} \\
    \holim_n \RHom(U_nF_1,U_nG) \ar{rr} && \holim_n \RHom(U_nF_0,U_nG) \ar[from=uu,crossing over]
 \end{tikzcd}\]
Since the contravariant $\RHom$-functor turns homotopy pushouts into homotopy pullbacks the upper horizontal square is a
homotopy pullback. The lower horizontal square is a homotopy pullback because the truncation $U_n$ preserves homotopy pushouts (and hence $U_nF$ is the homotopy pushout of $U_nF_1 \leftarrow U_nF_0 \rightarrow U_nF_2$) and homotopy limits preserve homotopy pullbacks. We can thus regard this cube as a morphism between
homotopy pullback squares. This morphism induces a weak equivalence in three columns
\[ \RHom(F_i,G) \rightarrow \holim_n \RHom(U_nF_i,U_nG) \]
and thus in the fourth column as well.
\end{proof}

We want to apply this machinery to the setting of dendroidal spaces.
To do so we need to write the indexing category $\Omega$ of trees as an increasing union
of full subcategories
\[
\Omega\lg 0\rg \subset  \Omega\lg 1\rg \subset \Omega\lg 2\rg \subset \Omega\lg 3 \rg \ldots \subset \Omega.
\]
A natural choice for a filtration of $\Omega$ comes from the observation that every finite tree has a unique maximal valence among all its vertices.
We will thus filter $\Omega$ by the maximal valence of the vertices.

\begin{defi}
 For $n\ge 0$ let $\Omega\lg n\rg$ denote the full subcategory of $\Omega$ on trees without vertices of valence $n+2$ or higher.
 An \emph{$n$-truncated dendroidal space} is a contravariant functor from $\Omega\lg n\rg$ to the category $\sSet$ of simplicial sets. The restriction
 functor along the inclusion $\Omega\lg n\rg \hookrightarrow \Omega$ will be denoted $U_n$.
\end{defi}

The categories $\Omega\lg n\rg$ have the property that their direct limit is the entire category $\Omega$.
We can map the derived mapping space $\RHom(X,Y)$ between two dendroidal spaces to a tower:
 \[
\xymatrix@R=12pt@C=15pt{
 && \rule{12mm}{0mm}\vdots\rule{12mm}{0mm} \ar[d]  \\
 && \RHom(U_3X,U_3Y) \ar[d] \\
 && \RHom(U_2X,U_2Y) \ar[d] \\
\RHom(X,Y) \ar[rr] \ar@/^1.1pc/[urr] \ar@/^1.8pc/[uurr] \ar@/^2.6pc/[uuurr] && \RHom(U_1X,U_1Y)
}
\]
of derived mapping spaces between their truncations. The previous discussion implies the convergence of this tower.

\begin{cor} \label{corfilt}
 For every pair $X$ and $Y$ of dendroidal spaces this tower converges, i.e., the map
 $\RHom(X,Y) \rightarrow \holim_n \RHom(U_nX,U_nY)$
is a weak homotopy equivalence.
\end{cor}
Of special interest will be the mapping space between the little disk operads introduced in \ref{littlediskdef}.
Here $X$ and $Y$ are dendroidal spaces weakly equivalent to nerves of operads of type $E_n$ and $E_m$ respectively.

\smallskip
We look for a description of the layers in the tower, i.e., the homotopy fibers of the forgetful map(s)
$\RHom(U_nX,U_nY)\to \RHom(U_{n-1}X,U_{n-1}Y)$.
There is such a description in the setting of \emph{$1$-reduced} dendroidal spaces. A dendroidal space $X$ is $1$-reduced if
$X(\eta)$, $X(\corl_0)$ and $X(\corl_1)$ are contractible spaces.
These correspond to monochromatic operads having contractible spaces in degrees $0$ and $1$.

\begin{defi}
Let $\Omega_{\cl}\subset \Omega$ be the full subcategory whose objects are the trees without any leaves, and
let $\iota\co \Omega_{\cl} \hookrightarrow \Omega$
be the inclusion functor. Objects of $\Omega_{\cl}$ will be called
\emph{closed trees}. We abbreviate $\odSet:=\Fun(\Omega_\cl,\Set)$ and $\sodSet:=\Fun(\Omega_\cl,\sSet)$.
Objects of these categories will be called \emph{closed dendroidal sets} and \emph{closed dendroidal spaces}, respectively. \\
The full subcategory $\Omega_{\cl}\cap\Omega\lg n\rg$ of $\Omega_{\cl}$ will be denoted $\Omega\lg n\rg_{\cl}$.
Its simplicial presheaves will be called \emph{$n$-truncated closed dendroidal spaces} and their category denoted $\sodSet\lg n\rg$.
\end{defi}

\begin{rema} Morphisms in $\Omega_{\cl}$ are much easier to understand than morphisms in $\Omega$. Recall that an object
of $\Omega_\cl$ is a finite partially ordered set $T$ (whose elements can be called edges)
subject to some conditions. There is no need to specify a set of leaves,
subset $L$ of $T$, because we are assuming that it is empty. A morphism from $T_0$ to $T_1$ in $\Omega_\cl$ was defined to be
a morphism of operads $\Omega(T_0)\to \Omega(T_1)$. But this boils down to a map $f\co T_0\to T_1$ which preserves the partial
order relation $\le$ and which preserves \emph{independence}. That is to say, if $x,y\in T_0$ are distinct and neither $x\le y$
nor $y\le x$ takes place, then none of $f(x)\le f(y)$, $f(y)\le f(x)$ takes place. \\
It does not follow that such an $f$ is injective. But it does follow that for every $z\in T_1$ the preimage $f^{-1}(z)$
is a linearly ordered subset of $T_0$. Moreover, if $x\in T_0\smin f^{-1}(z)$ satisfies $x\ge y_0$ for some $y_0\in f^{-1}(z)$, then it
satisfies $x\ge y$ for all $y\in f^{-1}(z)$. (Suppose not; then there is $y_1\in f^{-1}(z)$ such that $x$ and $y_1$ are independent;
but $f(x)> f(y_0)=f(y_1)$, contradiction since $f$ preserves independence.)
\end{rema}

\begin{defi}
 The \emph{closed $n$-corolla} $\ccorl_n$ is the unique tree in $\Omega_\cl$ with one vertex of valence $n+1$ and $n$ vertices of valence $1$.
(Just below: an artist's impression of $\ccorl_5$.)
\end{defi}
\[
\xymatrix@C=5pt@R=15pt{
\bullet \ar@{-}[drr] & \bullet \ar@{-}[dr]  & \bullet \ar@{-}[d] & \bullet \ar@{-}[dl]  & \bullet \ar@{-}[dll]  \\
&&  \bullet \ar@{-}[d] \\
&&
}
\]

\begin{lem} \emph{\cite[Lem 7.12]{BW15}}
 For all $1$-reduced monochromatic topological operads $P$ and $Q$ the restriction map
$
 \RHom(N_dP,N_dQ) \rightarrow \RHom(\iota^*N_dP,\iota^*N_dQ)
$
is a weak equivalence.
\end{lem}

In this restricted setting we can define levelwise boundaries and coboundaries generalizing the levelwise
boundaries in the description of the Fulton-MacPherson operad of \ref{FMop}.

\begin{defi} \label{defi-bocobo}
 Let $X \in \sdSet$ be a $1$-reduced dendroidal Segal space. In this definition we only
 use the restriction of $X$ to $\sodSet$.
 The \emph{$n$-th operadic boundary object} is the homotopy colimit
 \[
  \bound_n X := \hocolim_{(S,f) \in \ccorl_n/\Omega\lg n-1\rg_\cl} X_S.
 \]
 (This is a homotopy colimit of a contravariant functor, the functor which takes $(f \colon \ccorl_n \rightarrow S)$ to $X_S$.)
The \emph{$n$-th operadic coboundary object} is the homotopy limit
\[
  \cobound_n X := \holim_{(S,f) \in \Omega\lg n-1\rg_\cl/\ccorl_n} X_S.
 \]
 Both spaces come with an obvious $\Sigma_n$-action. There is a natural $\Sigma_n$-map from $\bound_n X$ to $X_{\ccorl_n}$
 induced by the various $f$ in pairs $(S,f)$, and similarly there is a natural $\Sigma_n$-map
from $X_{\ccorl_n}$ to $\cobound_n X$ induced by the various $f$ in pairs $(S,f)$.
The functor
\begin{align*}
  J_n \colon \sdSet \rightarrow \Fun(\Sigma_n \times [2],\sSet)  \\
	X \mapsto ( \bound_nX \rightarrow X_{\ccorl_n} \rightarrow \cobound_nX)
\end{align*}
sends a reduced dendroidal Segal space to the diagram consisting of these two maps. By composing the two maps we obtain
\begin{align*}
\partial J_n \colon \sdSet \rightarrow \Fun(\Sigma_n \times [1],\sSet)  \\
	X \mapsto ( \bound_nX \rightarrow \cobound_nX)
\end{align*}
\end{defi}
In other words, $\partial J_n=J_n\circ \rho$ where $\rho\co \Sigma_n\times[1]\to \Sigma_n\times[2]$ is induced
by the order-preserving injection $[1]\to [2]$ which does not take the value 1.
\begin{exam} \label{FMexample}
 Recall the Fulton-MacPherson operad $FM_k$ which was briefly introduced in \ref{FMop}. Then the space $\bound_nN_dFM_k$
 has the homotopy type of the boundary of the compact manifold-with-boundary $FM_k(n)$. Informally this can be seen because all the embeddings of substrata in the stratification of $FM_k(n)$ are cofibrations
 and thus the homotopy colimit in the definition of the operadic boundary object is weakly equivalent to the colimit which is exactly the boundary
 $\partial FM_k(n)$. It follows that we can model the map $\bound_nN_dE_k \rightarrow E_k(n)$ (for the little $k$-disk operad $E_k$)
by the inclusion $\partial FM_k(n) \hookrightarrow FM_k(n)$. In the following paragraph we give a more detailed argument.
\end{exam}

\begin{proof}
In this proof we will use the notation of \ref{defisubtree} and \ref{duggerstuff}.
To show this claim we first want to reduce the indexing category $\ccorl_n/\Omega\lg n-1\rg_\cl$ to a smaller one.
Every morphism $f \colon \ccorl_n \rightarrow T$ for $T \in \Omega\lg n-1\rg_\cl$
factors uniquely through a maximal subtree $T_0$ of $T$ with exactly $n$ outermost edges and all of them (as well as the root of $T_0$)
in the image of $f$. (By \emph{subtree} we mean something connected, so that if $x,z\in T_0$ and $y\in T$ satisfies $x\le y\le z$,
then also $y\in T_0$.) Let $I$ denote the subcategory of $\ccorl_n/\Omega\lg n-1\rg_\cl$ on all pairs $(S,g)$
such that $S$ has exactly $n$ outermost edges and all of them as well as the root are in the image of $g$.
We have just seen that the inclusion functor for this subcategory has a right adjoint. Therefore the inclusion
\[
\hocolim_{(S,f) \in I} FM_k(S) \longrightarrow  \bound_nN_dFM_k
\]
(where $FM_k(S)$ is short for $(N_dFM_k)_S$) is a weak equivalence by \cite[Thm.6.7]{Du}. Note in passing that the information provided
by the $f\co \ccorl_n\to S$ in a pair $(S,f)$ amounts to nothing more than a labeling of the outermost edges of $S$ with labels $1,2,\dots,n$.
We take this as an excuse for writing $S$ instead of $(S,f)$, but the labeling of the outermost edges remains important and must be
remembered. \\
Let $I^s$ denote the full subcategory of $I$ on those objects $S$ which have no vertices of valence $2$.
The inclusion $I^s\to I$ has a left adjoint $\sh\co I\to I^s$ and the unit morphisms for this adjunction
induce isomorphisms $FM_k(\sh(S))\to FM_k(S)$, for $(S,f)$ in $I$. By \cite[Thm.6.16.]{Du} the inclusion
\[ \hocolim_{S \in I^s}FM_k(S) \longrightarrow  \hocolim_{S \in I} FM_k(S)   \]
is a weak equivalence. ---
Recall the category $\Psi_n$ from the definition \ref{FMop} of the Fulton-MacPherson operad and let $\Psi^-_n$
denote the full subcategory of $\Psi_n$ on all trees not equal to the $n$-corolla.
This category $\Psi^-_n$ is equivalent to $I^{sh}$. We can therefore view $S\mapsto FM_k(S)$ as a
functor on $\Psi^-_n$. It remains only to show that the map
\[
 \hocolim_{S \in \Psi^-_n}FM_k(S) \rightarrow \partial FM_k(n)
\]
from the homotopy colimit to the actual colimit of this functor is a homotopy equivalence.
The plan is to show that this functor $FM_k \colon \Psi^-_n \rightarrow \sSet/ FM_k(n)$ is projectively cofibrant.
The category $\Psi^-_n$ is directed in the sense that
there is a faithful functor $\Psi^-_n \rightarrow \mathbb{N}$.
This is trivial since $\Psi^-_n$ is a finite poset, but here we have a preferred choice: the map which to every tree in $\Psi_n$
associates the number of its vertices.
Hence $\Psi^-_n$ becomes a Reedy category by defining the degree of a tree to be the negative of its number of vertices. Then every non-identity
morphism in $\Psi^-_n$ raises this degree. Let $M$ be some model category. A diagram $D \colon \Psi^-_n \rightarrow M$ is Reedy cofibrant if all its
latching maps are cofibrations. But by \cite[Thm 13.12]{Du} the Reedy model structure and the projective model
structure agree on upwards-directed Reedy categories and thus $D$ is Reedy cofibrant if and only if it is projectively cofibrant.
For every projectively cofibrant diagram its homotopy colimit is weakly equivalent to the actual colimit.
We thus want to show that $FM_k$ is Reedy cofibrant as a functor on $\Psi^-_n$. Let $T \in \Psi^-_n$ be a labelled tree.
The latching object $\textup{Lat}_T(FM_k)$ is the colimit over all maps $FM_k(S) \rightarrow FM_k(T)$ with $S \neq T$. But using the description of
the stratification of $FM_k(n)$ given in \ref{FMop} we see that
this map is just the inclusion of a union of substrata of the (closure of the) stratum corresponding to $T$ and thus a cofibration.
\end{proof}

\begin{thm} \label{mainthm}
Let $X$ and $Y$ be $1$-reduced dendroidal Segal spaces.
Then the following square of specialization or restriction maps is a homotopy pullback square:
\[
\begin{tikzcd}
\RHom(X \vert_{{\Omega\lg n\rg_\cl}},Y \vert_{{\Omega\lg n\rg_\cl}}) \arrow[d] \arrow[r] & \RNat(J_nX,J_nY) \arrow[d]  \\
\RHom(X\vert_{\Omega\lg n-1\rg_\cl},Y\vert_{\Omega\lg n-1\rg_\cl}) \arrow[r]& \RNat(\partial J_nX, \partial J_nY)
\end{tikzcd}
\]
\end{thm}
In the remainder
of this section we will make a reduction (by means of proposition \ref{restrictiontodext}) of this theorem to an easier statement.
The idea is to factor the inclusion $\Omega\lg n-1\rg_\cl \rightarrow \Omega\lg n\rg_\cl$ through certain intermediate subcategories
$\acat$ and $\bcat$.

\begin{defi}
Let $\acat$ be the full subcategory of $\Omega_\cl$ on
$\Omega\lg n-1\rg_\cl$ and the closed $n$-corolla $\ccorl_n$\,. Let $\bcat$ be the (slightly larger) full subcategory of
${\Omega\lg n\rg_\cl}$ on all objects of $\Omega\lg n-1\rg_\cl$
and all \emph{extended} (closed) $n$-corollas. These
are the objects of $\Omega_\cl$ which are connected to $\ccorl_n$ by a sequence of degeneracies.
For $n \neq 1$ they have a unique vertex of valence $n+1$ and only vertices of valence $2$ and $1$ otherwise.
\end{defi}

\begin{prop} \label{restrictiontodext}
 Let $X, Y \in \sodSet$ be restrictions of $1$-reduced dendroidal Segal spaces.
 Then the restriction map
 \[
  \RHom(X\vert_{{\Omega\lg n\rg_\cl}},Y\vert_{{\Omega\lg n\rg_\cl}}) \rightarrow
  \RHom(X\vert_{\acat},Y\vert_{\acat})
 \]
is a weak equivalence.
\end{prop}

In the following proofs we will also need the notion of subtree of a given tree $T$. This has already been used in \ref{FMexample}.
\begin{defi} \label{defisubtree}
Let $T$ be a tree in $\Omega_\cl$. A \emph{subtree} $S$ of $T$ consists of a subset of edges of $T$ such that the resulting graph is connected.
In this case $S$ should be understood as an object of $\Omega_\cl$ such that the inclusion $S \subset T$ is a morphism in $\Omega_\cl$.
\end{defi}

For example the closed $k$-corolla $\ccorl_k$ can be realized as a
subtree of the closed $n$-corolla $\ccorl_n$ (in $\binom{n}{k}$ different ways, assuming $k \neq 0$).
\begin{rema} \cite[Chapter 6]{Du} \label{duggerstuff}
 Let $\alpha \colon I \rightarrow J$ be a functor between small categories. For any $j \in J$ let $(j \downarrow \alpha)$ denote the category whose objects are pairs
$(i,f \colon j \rightarrow \alpha(i))$ and morphisms $(i,f) \rightarrow (i',f')$ are given by commutative triangles
\[
 \begin{tikzcd}
 j \arrow[r,"f"] \arrow[rd,"f'"] & \alpha(i) \arrow[d] \\
 \ & \alpha(i').
 \end{tikzcd}
\]
The functor $\alpha$ is called \emph{homotopy terminal} if for every $j$ the category $(j \downarrow \alpha)$ has a contractible
classifying space. Homotopy
terminal functors can be used to simplify homotopy colimits. More precisely for any homotopy terminal $\alpha$ and any diagram $X \colon J \rightarrow \sSet$
there is a natural weak equivalence
\[
 \hocolim_I \alpha^*X \rightarrow \hocolim_J X.
\]
There is a dual notion of a \emph{homotopy initial} functor $\beta \colon I \rightarrow J$. It has the property that all
overcategories $(\beta \downarrow j)$, defined dually to the undercategories above, are non-empty and contractible. In this case there is a
natural weak equivalence
\[
 \holim_I \beta^*Y \leftarrow \holim_J Y
\]
for all diagrams $Y \colon J \rightarrow \sSet$.
\end{rema}
In the following we write $\rran_F$ and $\rlan_F$ respectively for derived right and left Kan extensions along a functor $F$.
\begin{lem}
 Let $X$ in $\sodSet$ be a $1$-reduced dendroidal Segal space.
 Let $\phi$ denote the inclusion of $\bcat$ in $\Omega\lg n\rg_\cl$.
 Then the derived unit morphism
\[
 X\vert_{{\Omega\lg n\rg_\cl}} \rightarrow \rran_{\phi}~\phi^*\big(X\vert_{{\Omega\lg n\rg_\cl}}\big)
\]
is a weak equivalence.
\end{lem}
\begin{proof} We allow ourselves to write $\rran_{\phi}~\phi^*X$ instead of the more complicated expression in the statement.
By definition we have
\[
 (\rran_{\phi}~\phi^*X)_T \simeq \holim_{(f \colon S \rightarrow T)  \in \bcat/T} X_S\,.
\]
(We may also write $(S,f)$ instead of $(f \colon S \rightarrow T)$.)
It suffices to show that $\rran_{\phi}~\phi^*X$ has the Segal property. Indeed the unit maps $X_T \rightarrow (\rran_{\phi}~\phi^*X)_T$
are weak equivalences for all (closed) corollas $T$ in $\Omega\lg n\rg_\cl$. This follows from the fact that all these corollas are already in $\bcat$.
We can furthermore assume $\rran_{\phi}~\phi_*X$ to be projectively fibrant by choosing
a fibrant model for a homotopy limit of Kan complexes.  \\
We want to replace the indexing category $\bcat/T$ by an easier subcategory $\textsf{C}$ such that its inclusion functor is homotopy terminal.
The set of objects of $\textsf{C}$ is the set of subtrees as defined in \ref{defisubtree} of
$T$ understood as pairs $(A,a)$ of subtrees $A$ (which are objects of $\bcat$ in their own right) with fixed inclusions $a \colon A \rightarrow T$. Maps $\tau \colon (A,a) \rightarrow (B,b)$ are given by commutative triangles
\[
\begin{tikzcd}
 T   & & A \arrow[ll,"a"] \arrow[dl,"\tau"] \\
 & B \arrow[ul,"b"]. &
\end{tikzcd}
\]
The inclusion of $\textsf{C}$ in $\bcat/T$ has a left adjoint. (This works only because we are using $\bcat$
instead of the smaller $\acat$.) Consequently the inclusion of $\textsf{C}$ into $\bcat/T$ is indeed homotopy terminal.
It follows from the contravariant version of \cite[Thm. 6.12]{Du} that
the forgetful projection
\[
\holim_{(S,f) \in \bcat/T} X_S  \longrightarrow \holim_{(A,a) \in \textsf{C}} X_A
\]
is a weak equivalence.   \\
Let $X'$ denote the induced functor $\textsf{C} \rightarrow \sSet$. We want to prove that the decompositions of the $X'_S$ given by the Segal property
of $X$ are natural in this reduced setting, i.e. every morphism $S \rightarrow S'$ in $\textsf{C}$ induces a map $X'_{S'} \rightarrow X'_{S}$ which
respects the product decomposition and can be defined factorwise. Let $g \colon S \rightarrow S'$ be any morphism in $\textsf{C}$.
It is an inclusion of a subtree. Let $\{v_1, \ldots, v_k\}$ be the set of vertices of $T$.
Let $\vert v_j \vert_S$ denote the number of inputs of $S$ at $v_j$. Since we assumed $X$ to be $1$-reduced and to satisfy the Segal property
we know that the morphism
\[
 X_S \rightarrow X^{\Sc[S]} \cong X_{\ccorl_{\vert v_1 \vert_S}} \times \ldots \times X_{\ccorl_{\vert v_k \vert_S}}
\]
induced by the Segal core inclusion $\Sc[S] \rightarrow \Omega_\cl[S]$ is a trivial Kan fibration. By functoriality of $X$ the square
\[
 \begin{tikzcd}
  X_{S'} \arrow[r] \arrow[d] & X_S \arrow[d] \\
	X^{\Sc[S']} \arrow[r] & X^{\Sc[S]}
 \end{tikzcd}
\]
commutes. We only have to show that the maps $X^{\Sc[S']} \rightarrow X^{\Sc[S]}$ can be defined factorwise. But this is immediate because the map
$\Sc[S] \rightarrow \Sc[S']$ is induced by a morphism $S\to S'$ over $T$. Let $X''$ denote
the functor $\textsf{C} \rightarrow \sSet$
defined by the composition $X' \circ \Sc$. We have a natural transformation $X' \rightarrow X''$ which is a levelwise trivial fibration. Hence
\[
 \holim_{S \in \textsf{C}} X_S \simeq \holim_{S \in \textsf{C}} X''_S.
\]
\end{proof}

\begin{lem}
 Let $X \in \sodSet$ be a $1$-reduced dendroidal Segal space. Write $\psi$ for the inclusion $\psi\co \acat \rightarrow \bcat$. The derived
 counit map
 \[
  \rlan_{\psi}~\psi^*(X\vert_{\bcat}) \rightarrow X\vert_{\bcat}
 \]
is a weak equivalence.
\end{lem}

\begin{proof}
We have to show that this morphism induces a weak equivalence of simplicial sets at any tree $T$. For trees in $\acat$
there is nothing to show. So let $T$ be an extended (closed) $n$-corolla. The space $( \rlan_{\psi}~\psi^*(X\vert_{\bcat}))_T$ is the homotopy
colimit
\[
 \hocolim_{(S,f) \in (T / \acat)} X_S.
\]
Here $(S,f)$ is short for $f \colon T \rightarrow S$. We want to find an easier category $\textsf{C}$ and a homotopy initial functor
$\textsf{C} \rightarrow T/\acat$.
Let the set of objects of $\textsf{C}$ be the set of all pairs $(S,f)$ such that $S$ is
a tree in $\acat$ with exactly $n$ vertices of valence $1$ and $f\co T \rightarrow S$ is a morphism such every outermost edge (including the root)
of $S$ is in the image of $f$.
A morphism $\tau \colon (S_0,f) \rightarrow (S_1,g)$ is given by a commutative triangle
\[
\begin{tikzcd}
 T \arrow[rr,"f"] \arrow[dr,"g"] & & S_0 \arrow[dl,"\tau"] \\
 & S_1 &
\end{tikzcd}
\]
in $\bcat$.
The category $\textsf{C}$ is a full subcategory of
$T/\acat$.
The inclusion functor $\textsf{C} \rightarrow T/\acat$ has a right adjoint.
This implies that the inclusion is homotopy initial.
By the contravariant version of \cite[Thm 6.7.]{Du} we get a weak equivalence
\[
\hocolim_{(S,f) \textup{ in }\textsf{C}} \phi^*X_S \xrightarrow{~~\simeq~~} \rlan_{\psi}\psi^*\phi^*X_T
\]
In a second step we replace the indexing category $\textsf{C}$ by another even easier one.
Let $\textsf{C}^{\sh}$ be the full subcategory of $\textsf{C}$ on the pairs $(S,f)$ such that $S$ has no vertices of valence $2$.
The inclusion $\alpha \colon \textsf{C}^{\sh} \rightarrow \textsf{C}$ has a left adjoint $\beta$. Clearly $\beta\alpha=\Id$
and the counit transformation for this adjunction is the identity.
Let $\tau \colon \Id_\textsf{C} \rightarrow \alpha\beta$ denote the unit transformation for this adjunction.
Because $X$ has the Segal property and is $1$-reduced the induced map
$\tau^*_S \colon X_{\beta(S)} \rightarrow X_S$ is a weak equivalence for every $(S,f)$ in $\textsf{C}$.
The contravariant version of \cite[Prop 6.16]{Du} thus implies that the canonical morphism
\[
 \hocolim_{(S,f) \textup{ in }\textsf{C}^{\sh}} X_{S} \rightarrow \hocolim_{(S,f) \textup{ in }\textsf{C}} X_S
\]
is a weak equivalence. The category $\textsf{C}^{\sh}$ has an initial object given by the closed $n$-corolla $\ccorl_n$
together with the unit map $T \rightarrow \ccorl_n$. Thus
\[
 X_{\ccorl_n} \simeq \hocolim_{(S,f) \textup{ in }\textsf{C}^{\sh}} X_{S}.
\]
By our assumptions on $X$ this concludes the proof.
\end{proof}

Thus theorem \ref{mainthm} reduces to:
\begin{thm} \label{zweieinszwanzig}
Let $X$ and $Y$ be $1$-reduced dendroidal Segal spaces. The following square of specialization maps is a homotopy pullback:
\[
\begin{tikzcd}
\RHom(X\vert_{\acat},Y\vert_{\acat}) \arrow[d] \arrow[r] & \RNat(J_nX,J_nY) \arrow[d]  \\
\RHom(X\vert_{\Omega\lg n-1\rg_\cl},Y\vert_{\Omega\lg n-1\rg_\cl}) \arrow[r]& \RNat(\partial J_nX, \partial J_nY)
\end{tikzcd}
\]
\end{thm}

This now follows from \ref{ExciThm}. We note that the maps in the square need careful definitions. They will be given in section~\ref{sec-exci}.
The right hand column of this homotopy pullback square can be modified as explained in the following remark.

\begin{rema} \label{remarkThomas}
 The following square is a homotopy pullback square (and we switch from $n$ to $k$):
 \[
  \begin{tikzcd}
   \RNat(J_kX,J_kY) \arrow[r] \arrow[d] & \RMap_{\Sigma_k}(X_{\ccorl_k},Y_{\ccorl_k}) \arrow[d] \\
   \RNat(\partial J_kX, \partial J_kY) \arrow[r] & \RMap_{\Sigma_k}(\bound_kX \rightarrow X_{\ccorl_k}\, , \,Y_{\ccorl_k} \rightarrow \cobound_k Y)
  \end{tikzcd}
\]
The horizontal maps are the obvious forgetful maps. The right hand vertical arrow is explained by the following diagram:
\[
 \begin{tikzcd}
 \bound_kX \arrow[d] \arrow[rd,dashrightarrow] & \ \\
 X_{\ccorl_k} \arrow[r] \arrow[rd,dashrightarrow] & Y_{\ccorl_k} \arrow[d] \\
 \   & \cobound_kY
 \end{tikzcd}
 \]
 (A short argument for the homotopy pullback property: compare the horizontal homotopy fibers.)
Therefore the main theorem as stated in the introduction is equivalent to the statement that the square
\[
\begin{tikzcd}
\RHom(U_kX,U_kY) \arrow[d] \arrow[r] & \RMap_{\Sigma_k}(X_{\ccorl_k},Y_{\ccorl_k}) \arrow[d]  \\
\RHom(U_{k-1}X,U_{k-1}Y) \arrow[r]& \RMap_{\Sigma_k}(\bound_kX \rightarrow X_{\ccorl_k}\,
, \,Y_{\ccorl_k} \rightarrow \cobound_k Y)
\end{tikzcd}
\]
is a homotopy pullback square. Note that $U_kX$ can be read as $X\vert_{\Omega\lg k\rg_\cl}$ etc.
The interesting ``news'' here is that for a homotopical description of $k$-th layer in the tower, we
need to know only the diagrams $\bound_kX \to X_{\ccorl_k}$ (for the source) and $Y_{\ccorl_k}\to \cobound_kY$
(for the target). As before, they are diagrams of spaces with an action of $\Sigma_k$\,.
\end{rema}

\section{Excision in categories} \label{sec-exci}
 % Main appendix title
\label{AppendixEx} % For referencing this appendix elsewhere, use \ref{AppendixA}

Let $I$ be a small category.
In this section we will investigate how spaces of natural transformations between two $I$-shaped diagrams of simplicial sets change if we remove or
add a new object (with morphisms) from or to $I$. Throughout this section all model categories (of $I$-diagrams) are equipped with the projective model structure. We choose
functorial fibrant and cofibrant replacements. (Their existence follows from the small object argument.) As a model for derived mapping spaces
$\RNat(F,G)$
we
choose the simplicial mapping space between the cofibrant and fibrant replacements $\Map(F^c,G^f)$.\\

\subsection{Boundaries and coboundaries revisited}
From now on we assume the indexing category $I$ to be skeletal (distinct objects are not isomorphic). Let $x$ be an object of $I$ such that
no object $i$ (distinct from x) which admits a morphism to $x$ receives a morphism from $x$. We also require that every
endomorphism of $x$ be an automorphism.
This is for example the case if $I$ is a direct category. The most important example for us is the case
where $I = (\acat)^{\op}$ and $x = \ccorl_n$.

The functor $J_n$ of definition~\ref{defi-bocobo} generalizes in an obvious way
to a functor $J_x$ for arbitrary $I$-shaped diagrams of spaces.
\begin{defi}
Let $I$ and $x$ be as above.
 For any diagram $F \in \Fun(I,\sSet)$ we define
\[
 \bound_x(F) := \hocolim_{(f \colon y \rightarrow x) \textup{ in } I/x} F(y)
\]
and
\[
 \cobound_x(F) := \holim_{(g \colon x \rightarrow y) \textup{ in } x/I} F(y).
\]
These spaces serve as a replacement for the operadic boundary and coboundary space, respectively. They come with a natural action of the automorphism
group of $x$. The functor
\[
J_x \colon \Fun(I,\sSet) \rightarrow \sSet^{\Aut(x) \times [2]}
\]
 is now defined by sending a diagram $F$ to the sequence
\[
 \bound_x(F) \rightarrow F(x) \rightarrow \cobound_x(F).
\]
We note this is to be viewed as a functor from $\Aut(x) \times [2]$ to simplicial sets.
Then $\partial J_x$ is defined by sending $F$ to the subsequence
\[
 \bound_x(F) \rightarrow \cobound_x(F)
\]
with the same equivariance properties. Both functors respect levelwise weak equivalences. There is a functor $\rho$
from $\sSet^{\Aut(x) \times [2]}$ to $\sSet^{\Aut(x) \times [1]}$ given by omitting the middle object and composing the two morphisms;
we can write $\rho(J_x)$ instead of $\partial J_x$.
% We note that under our assumptions on $I$ and $x$ there is no morphism $f$ from any
% object $a$ in the category indexing $\cobound_x(F)$ to any $b$ in the category indexing $\bound_x(F)$.

\end{defi}

We obtain a commutative diagram of mapping spaces, where the superscripts $c$, $f$ and $r$ have the following meaning:
$c$ is for cofibrant replacement, $f$ is for fibrant replacement and $r$ is for restriction from $I$ to $I\smin x$, full subcategory of $I$.
\[
\begin{tikzpicture}[baseline= (a).base]
\node[scale=.85] (a) at (0,0){
 \begin{tikzcd}
 \Map(F^c,G^f) \arrow[d] \arrow[r] & \Map(J_xF^c,J_xG^f) \arrow[r] \arrow[d] & \Map((J_xF^c)^c,(J_xG^f)^f) \arrow[d] \\
 \Map((F^c)^r,(G^f)^r) \arrow[d] \arrow[r] & \Map(\partial J_xF^c,\partial J_xG^f) \arrow[r] \arrow[d] & \Map((\partial J_xF^c)^c,(\partial J_xG^f)^f) \arrow[d] \\
 \Map(((F^c)^r)^c,((G^f)^r)^f) \arrow[r] & \Map(\partial J_x(((F^c)^r)^c),\partial J_x(((G^f)^r)^f)) \arrow[r] & \Map((\partial J_x(F^c)^c)^c,(\partial J_x(G^f)^f)^f).
  \end{tikzcd}
  };
\end{tikzpicture}
\]
We will informally abbreviate the outer square to
\[
\begin{tikzcd}
\RNat(F,G) \arrow[d] \arrow[r] & \RNat(J_xF,J_xG) \arrow[d]  \\
\RNat(F^r,G^r) \arrow[r]& \RNat(\partial J_xF, \partial J_xG).
\end{tikzcd}
\]
\emph{Justification:} There are natural weak equivalences
\begin{alignat*}{3}
 \RNat(F^r,G^r)  & \longrightarrow && \Map(((F^c)^r)^c,((G^f)^r)^f) \\
 \RNat(J_xF,J_xG) & \longrightarrow && \Map((J_xF^c)^c,(J_xG^f)^f) \\
 \RNat(\partial J_xF,\partial J_xG) & \longrightarrow  && \Map((\partial J_x(F^c)^c)^c,(\partial J_x(G^f)^f)^f)
\end{alignat*}
given by suitable pre- and postcompositions.

\begin{thm} \label{ExciThm}
 Let $I$ and $x \in I$ be as above. Let
 $F$ and $G$ be functors from $I$ to $\sSet$ and let $F^r$ and $G^r$ be their restrictions to $I\smin x$.
 Then the following square is a homotopy pullback:
\[
\begin{tikzcd}
\RNat(F,G) \arrow[d] \arrow[r] & \RNat(J_xF,J_xG) \arrow[d]  \\
\RNat(F^r,G^r) \arrow[r]& \RNat(\partial J_xF, \partial J_xG).
\end{tikzcd}
\]
\end{thm}

\subsection{D\'evissage}
We will prove theorem~\ref{ExciThm} in three steps using the principle we developed in \ref{principle}.
Thoughout these steps we will keep the notation of \ref{ExciThm}.

\begin{lem}
Let $F$ the homotopy pushout of $F_1 \leftarrow F_0 \rightarrow F_2$. If the statement \ref{ExciThm} holds for the $F_i$ then it holds for $F$.
\end{lem}
\begin{proof}
We observe that the two terms on the left hand side turn homotopy pushouts into homotopy pullbacks. It follows that each vertical homotopy
fiber in the left hand column for $F$ becomes
a homotopy pullback of the respective fibers for the $F_i$.
\\
To understand the right hand homotopy fibers we note that we can arrange the resulting spaces into a cube.
\[\begin{tikzcd}[row sep={40,between origins}, column sep={80,between origins}]
      & \RHom(J_xF,J_xG) \ar{rr}\ar{dd}\ar{dl} & & \RHom(J_xF_1,J_xG) \ar{dd}\ar{dl} \\
    \RHom(J_xF_2,J_xG) \ar[crossing over]{rr} \ar{dd} & & \RHom(J_xF_0,J_xG) \\
      & \RHom(\partial J_xF,\partial J_xG)  \ar{rr} \ar{dl} & & \RHom(\partial J_xF_1,\partial J_xG) \ar{dl} \\
    \RHom(\partial J_xF_2,\partial J_xG) \ar{rr} && \RHom(\partial J_xF_0,\partial J_xG) \ar[from=uu,crossing over]
\end{tikzcd}\]
We want to prove that this cube is homotopy cartesian. To do so we pick a point in the initial term $\RHom(\partial J_xF,\partial J_xG)$
of the lower square and thus in each term of the lower square. Then we obtain a square of vertical homotopy fibers and we want to show this square is
homotopy cartesian. \\
One of these vertical homotopy fibers, the homotopy fiber of
\[ \RHom(J_xF,J_xG) \rightarrow \RHom(\partial J_xF,\partial J_xG) \] consists of
$\Aut(x)$-equivariant lifts
\[
 \begin{tikzcd}
 \bound_xF \arrow[d] \arrow[r] & F(x) \arrow[d,dotted] \arrow[r] & \cobound_xF \arrow[d] \\
 \bound_xG \arrow[r] & G(x) \arrow[r] & \cobound_xG.
 \end{tikzcd}
\]
Thus a lift consists of an equivariant morphism $F(x) \rightarrow G(x)$, compatible homotopies $\bound_xF \times \Delta[1] \rightarrow G(x)$ and
$F(x) \times \Delta[1] \rightarrow \cobound_xG$ and a homotopy of homotopies $\bound_xF \times (\Delta[1]\times\Delta[1]) \rightarrow \cobound_xG$.
More formally the space of lifts is the total homotopy fiber of the square
\[
 \begin{tikzcd}
 \RHom(F(x),G(x)) \arrow[d] \arrow[r] & \RHom(F(x),\cobound_xG) \arrow[d]  \\ \RHom(\bound_xF,G(x)) \arrow[r] & \RHom(\bound_xF, \cobound_xG)
 \end{tikzcd}
\]
over the (three) basepoints determined by the basepoint in $\RHom(\partial J_xF,\partial J_xG)$ which we selected.
We obtain similar results for the other three vertical homotopy fibers by replacing $F$ with $F_i$. From these descriptions it is clear that the
square formed by the vertical homotopy fibers is homotopy cartesian. Therefore the cube is homotopy cartesian. \\
To conclude the statement we note that there is a morphism of homotopy cartesian cubes from
\[\begin{tikzcd}[row sep={40,between origins}, column sep={80,between origins}]
      & \RHom(F,G) \ar{rr}\ar{dd}\ar{dl} & & \RHom(F_1,G) \ar{dd}\ar{dl} \\
    \RHom(F_2,G) \ar[crossing over]{rr} \ar{dd} & & \RHom(F_0,G) \\
      & \RHom(F^r,G^r)  \ar{rr} \ar{dl} & & \RHom(F^r_1,G^r) \ar{dl} \\
    \RHom(F^r_2,G^r) \ar{rr} && \RHom(F^r_0,G^r) \ar[from=uu,crossing over]
\end{tikzcd}\]
to
\[\begin{tikzcd}[row sep={40,between origins}, column sep={80,between origins}]
      & \RHom(J_xF,J_xG) \ar{rr}\ar{dd}\ar{dl} & & \RHom(J_xF_1,J_xG) \ar{dd}\ar{dl} \\
    \RHom(J_xF_2,J_xG) \ar[crossing over]{rr} \ar{dd} & & \RHom(J_xF_0,J_xG) \\
      & \RHom(\partial J_xF,\partial J_xG)  \ar{rr} \ar{dl} & & \RHom(\partial J_xF_1,\partial J_xG) \ar{dl} \\
    \RHom(\partial J_xF_2,\partial J_xG) \ar{rr} && \RHom(\partial J_xF_0,\partial J_xG), \ar[from=uu,crossing over]
\end{tikzcd}\]
which we can view as a (homotopy cartesian) $4$-cube, or better, as a square of squares.
By assumption the three squares
\[
\begin{tikzcd}
 \RNat(F_i,G) \arrow[r] \arrow[d] & \RHom(J_xF_i,J_xG) \arrow[d] \\
   \RNat(F_i^r,G^r) \arrow[r] & \RHom(\partial J_xF_i,\partial J_xG)
\end{tikzcd}
\]
are homotopy cartesian. It follows that the square
\[
\begin{tikzcd}
 \RNat(F,G) \arrow[r] \arrow[d] & \RHom(J_xF,J_xG) \arrow[d] \\
   \RNat(F^r,G^r) \arrow[r] & \RHom(\partial J_xF,\partial J_xG)
\end{tikzcd}
\]
is homotopy cartesian as well.
\end{proof}

\begin{lem}
Let $F$ the levelwise disjoint union of functors $F_{\alpha}$. If the statement \ref{ExciThm} holds for the $F_{\alpha}$ then it holds for $F$.
\end{lem}
\begin{proof}
 The left hand side of the square is easy to understand. If $F = \coprod_{\alpha} F_{\alpha}$ then $\RNat(F,G) = \prod_{\alpha}\RNat(F_{\alpha},G)$ and
 similarly, since restriction preserves disjoint unions,
 $F^r = \coprod_{\alpha} F^r_{\alpha}$ and $\RNat(F^r,G^r) = \prod_{\alpha}\RNat(F^r_{\alpha},G^r)$. \\
 On the right hand side we have $\bound_xF = \coprod_{\alpha}\bound_xF_{\alpha}$.
 The proof follows like the previous one by comparison of the vertical homotopy fibers. Each left hand vertical homotopy fiber for $F$
 decomposes as a product of the corresponding left hand vertical homotopy fibers for the $F_{\alpha}$. \\
 As we have seen in the proof of the previous lemma each right hand vertical homotopy fiber is the total homotopy fiber of
\[
 \begin{tikzcd}
 \RHom(F(x),G(x)) \arrow[d] \arrow[r] & \RHom(\bound_xF,G(x)) \arrow[d] \\ \RHom(F(x),\cobound_xG) \arrow[r] & \RHom(\bound_xF,\cobound_xG)
 \end{tikzcd}
\]
over a compatible selection of three base points which we obtained by choosing a point in $\RNat(\partial J_xF,\partial J_xG)$. Both terms
$F(x)$ and $\bound_xF$ preserve disjoint unions and thus each right hand vertical homotopy fiber splits as a product. By assumption the squares
\[
 \begin{tikzcd}
 \RNat(F_\alpha,G) \arrow[r] \arrow[d] & \RHom(J_xF_\alpha,J_xG) \arrow[d] \\
   \RNat(F_\alpha^r,G^r) \arrow[r] & \RHom(\partial J_xF_\alpha,\partial J_xG)
\end{tikzcd}
\]
are homotopy cartesian.
\end{proof}

\begin{lem}
 The statement \ref{ExciThm} holds for $F$ representable.
\end{lem}
\begin{proof}
 Let $F = \Hom(y,-)$, in other words $F$ is (co-)represented by $y$. We will distinguish three different cases.
 First assume $y = x$. Then $\RHom(F,G) \simeq G(x)$ and the left hand vertical arrow becomes
\[
 \begin{tikzcd}
 G(x) \arrow[d] \\
 \holim_{x \rightarrow z} G(z) = \cobound_x G.
 \end{tikzcd}
\]
On the right hand side we have $J_xF = (\emptyset \rightarrow \Hom(x,x) \rightarrow \cobound_x F)$ and
consequently $\partial J_xF = (\emptyset \rightarrow \cobound_x F)$. Thus a point in the right hand vertical homotopy fiber consists
of a choice of an $\Aut(x)$-equivariant lift
\[
\begin{tikzcd}
 \Hom(x,x) \arrow[r,dotted] \arrow[d] & G(x) \arrow[d] \\
 \cobound_x F \arrow[r] & \cobound_x G.
\end{tikzcd}
\]
But since $\Hom(x,x)$ is the free $\Aut(x)$-space on a point (by assumption) these lifts are in a $1$-to-$1$ correspondence to non-equivariant lifts
of points in $\cobound_x G$ to $G(x)$. This space is homotopy equivalent to the left hand homotopy fiber and the induced map between the two is
a weak equivalence. Thus the square of \ref{ExciThm} is a homotopy pullback. \\
As a second case we assume that there is a morphism $y \rightarrow x$ but $y \neq x$.
In this case the left hand vertical morphism is homotopic to the identity
\[
 \begin{tikzcd}
  G(y) \arrow[d] \\
  G^r(y) = G(y).
 \end{tikzcd}
\]
On the right hand side we see that
\[
\bound_xF = \hocolim_{z \rightarrow x} \Hom(y,z) \simeq \Hom(y,x).
\]
This is because we can write $\hocolim_{z \rightarrow x} \Hom(y,z)$ as the classifying space
of the category of diagrams of the form $y \rightarrow z \rightarrow x$ with fixed $y$ and $x$.
That category of diagrams has a subcategory consisting of the diagrams
\[ y \xrightarrow{\Id} y \rightarrow x. \]
The inclusion of this subcategory has a right adjoint given by
\[
 (y \xrightarrow{f} z \xrightarrow{g} x) \mapsto (y \xrightarrow{\Id} y \xrightarrow{gf} x).
\]
Thus the morphism $\bound_xF \rightarrow F(x) = \Hom(y,x)$ is a weak equivalence.
We are thus looking for (equivariant) solutions of
\[
\begin{tikzcd}
\Hom(y,x) \arrow[r] \arrow[d,"="] & \bound_xG \arrow[d] \\
 \Hom(y,x) \arrow[r,dotted] \arrow[d] & G(x) \arrow[d] \\
 \cobound_x F \arrow[r] & \cobound_x G
\end{tikzcd}
\]
(the broken arrow, two primary homotopies and a secondary homotopy).
But the middle morphism is already determined to be the composition
\[
\RHom(y,x) \rightarrow \bound_xG \rightarrow G(x)
\]
and hence the right hand homotopy fiber is contractible as well.
\\
Now assume there is no morphism $y \rightarrow x$. Then $\bound_xF$ as well as $F(x)$ are empty sets.
In this case both vertical
morphisms are isomorphisms and the square is a homotopy pullback square. Thus the statement holds for all representable functors.
 \end{proof}

\begin{rema}
 In the case $I^{\op} = \acat$ and $x = \ccorl_n$ the object $\bound_xF$ is naturally weakly equivalent
 to $\hocolim_{\ccorl_n \rightarrow S} F(S)$, where we think of $F$ as contravariant and we only allow morphisms
 $\ccorl_n \rightarrow S$ in $\acat$ satisfying
 two conditions: The outer edges of $S$ (including the root) are in the image and $S$ has no vertices of valence $2$.
 This was essentially proved in \ref{FMexample}.
 Similarily we can restrict the homotopy limit $\holim_{S \rightarrow \ccorl_n}F(S)$ to the category of subtrees (as defined in \ref{defisubtree}) of $\ccorl_n$.
 Therefore
 the coboundary object is equivalent to a homotopy limit over a punctured $n$-cube.
\end{rema}

\section{Application} % Main chapter title

In theorem~\ref{mainthm} we have constructed a new and easier model for the homotopy fibers of our tower.
We still have to show that it is possible to obtain any information about the homotopy type of these layers from the new description.

\subsection{The derived mapping spaces of little cube operads}

In this section we will finally apply the machinery we developed to a concrete case. We will compute the connectivity of the layers of the tower for
$\RHom(N_dE_n,N_dE_{n+d})$. Beware that the $d$ in $N_d$ means dendroidal and everywhere else
the $d$ is used to denote the codimension. Rational versions of the results of this section were obtained by Fresse, Turchin and Willwacher in \cite[Chapter 10]{FTW18}.
\begin{rema}
 The derived mapping space $\RHom(N_dE_n\vert_{\Omega\lg 1\rg_\cl},N_dE_{n+d}\vert_{\Omega\lg 1\rg_\cl})$ is contractible.
\end{rema}

\begin{lem} \label{hodim}
 The pair $(E_n(k),\bound_kE_n)$ is homotopy equivalent to a CW pair $(X,Y)$ with no relative cells of dimension above $n(k-1)-1$.
\end{lem}
\begin{proof}
 Follows from the construction of the Fulton-MacPherson operad.
 Notably the inclusion of the boundary $\partial FM_n(k) \rightarrow FM_n(k)$ is
 a model for the operadic boundary inclusion map and the smooth manifold $FM_n(k)$ has dimension $n(k-1)-1$. We have shown this in \ref{FMexample}.
\end{proof}

\begin{lem} \label{connec}
 The map $E_n(k) \rightarrow \cobound_k(E_n)$
is $((k-1)(n-2)+1)$ connected.
\end{lem}
\begin{proof}
In \cite[Ex 6.2.9]{MV15} Munson and Volic show that the $k$-cube of ordered configu\-ration spaces
defined by $S \mapsto \Conf(S,M)$ for $S\subset\{1,2,\dots,k\}$ and a fixed $n$-dimensional manifold $M$
is $((k-1)(n-2)+1)$-cartesian. % looked it up in MV! but can we do better in the case $M=\RR^n$? Probably not.
\end{proof}

\begin{thm} Each homotopy fiber of
 \[
 \RNat(J_kE_n,J_kE_{n+d}) \rightarrow \RNat(\partial J_kE_n, \partial J_kE_{n+d})
 \]
 is $((k-1)(d-2)+1)$-connected.
\end{thm}
\begin{proof}
By remark \ref{remarkThomas} the homotopy fiber is equivalent to a total homotopy fiber of the square
\[
 \begin{tikzcd}
 \Map(E_n(k),E_{n+d}(k)) \arrow[d] \arrow[r] & \Map(E_n(k),\cobound_kE_{n+d}) \arrow[d]  \\ \Map(\bound_kE_n,E_{n+d}(k)) \arrow[r] & \Map(\bound_kE_n, \cobound_kE_{n+d})
 \end{tikzcd}
\]
By general principles the connectivity of such a total homotopy fiber is not less than the connectivity of
$\hofib(E_{n+d}(k) \rightarrow \cobound_kE_{n+d})$ minus the relative homotopical dimension of the inclusion $\bound_kE_n \hookrightarrow E_n(k)$.
By \ref{hodim} the first number is at least $(k-1)(n+d-2)$ and the second number is $n(k-1)-1$ by \ref{connec}.
The difference of these numbers turns out to be $(k-1)(d-2)+1$.
\end{proof}
Thus we get:
\begin{cor}
Assume $d \geq 2$.
 The derived mapping space $\RHom(E_n,E_{n+d})$ is $(d-1)$-connected. Furthermore all
 spaces of derived maps between their truncations are $(d-1)$-connected as well.
\end{cor}

\printbibliography[heading=bibintoc]

Math.~Institute, University of Münster, Einsteinstrasse 62, 48149 Münster, Germany \\
florian.goeppl@gmail.com \quad m.weiss@uni-muenster.de

%----------------------------------------------------------------------------------------

\end{document}